\documentclass[ps,preprint,dvips,twoside]{imsart}

\usepackage{amsmath,amssymb,amsfonts,amsthm,amstext,verbatim}
\usepackage{url,enumerate,color,graphicx,stmaryrd,psfrag}
\usepackage{chngcntr}
\usepackage{mdwlist}   
\usepackage{mathtools} 
\usepackage{wasysym}   
\RequirePackage[colorlinks,citecolor=blue,urlcolor=blue]{hyperref}
\usepackage{breakurl}

\doi{10.1214/11-PS185}
\pubyear{2011}
\volume{8}
\issue{0}
\firstpage{294}
\lastpage{332}

\startlocaldefs

\newcommand\sT{{\mathcal T}}

\newcommand\ol{\overline}
\newcommand\oo{\infty}
\newcommand\sm{\setminus}
\newcommand\De{\Delta}

\newcommand\rad{\text{\rm rad}}
\renewcommand\b{\beta}
\newcommand\si{\sigma}

\newcommand\g{\gamma}
\renewcommand\a{\alpha}
\newcommand\pd{\partial}

\newcommand\es{\varnothing}
\newcommand\Om{\Omega}
\newcommand\om{\omega}
\newcommand\Ga{\Gamma}
\newcommand\La{\Lambda}
\newcommand\Si{\Sigma}
\newcommand\sA{{\mathcal A}}

\newcommand\sS{{\mathcal S}}
\newcommand\sM{{\mathcal M}}
\newcommand\sL{{\mathcal L}}
\newcommand\pc{p_{\text{\rm c}}}
\newcommand\tc{{\text{\rm c}}}
\newcommand\de{\delta}
\newcommand\resp{respectively}
\newcommand\lra{\leftrightarrow}

\renewcommand\th{\theta}
\newcommand\bxp{box-crossing property}
\newcommand\bp{\mathbf{p}}
\newcommand\squ{{\square}}
\newcommand\tri{{\triangle}}
\newcommand\hex{{\hexagon}}

\newcommand\stt{star--triangle transformation}
\newcommand\Cv{C_{\text{\rm v}}}
\newcommand\Ch{C_{\text{\rm h}}}

\newcommand\lest{\le_{\text{\rm st}}}
\newcommand\bigmid{\,\big|\,}

\newcommand\sP{{\mathcal P}}

\newcommand\eps{\epsilon}
\newcommand\ups{\zeta}

\newcommand\compl[1]{#1^{\tc}}
\newcommand{\RR}{\mathbb{R}}     
\newcommand{\NN}{\mathbb{N}}     
\newcommand{\ZZ}{\mathbb{Z}}     
\newcommand{\PP}{\mathbb{P}}     
\newcommand\CC{\mathbb{C}}
\newcommand{\EE}{\mathbb{E}}     
\newcommand\HH{\mathbb{H}}
\newcommand\TT{\mathbb{T}}
\newcommand\LL{\mathbb{L}}
\newcommand{\comp}{\circ} 
\newcommand{\Su}{S^{\Yup}}
\newcommand{\Sd}{S^\Ydown}
\newcommand{\Tu}{T^\vartriangle}
\newcommand{\Td}{T^\triangledown}

\newcommand{\Ann}{\mathcal{A}}

\newcommand{\Lattice}{\mathbb{L}}

\def\mik{1}
\newcommand\cpsfrag[2]{\ifnum\mik=1\psfrag{#1}{#2}\fi}

\newcounter{mycount}

\newenvironment{numlist}{\begin{list}{\arabic{mycount}.}%
   {\usecounter{mycount}\labelwidth=1cm\itemsep 0pt}}{\end{list}}
\newenvironment{letlist}{\begin{list}{\rm(\alph{mycount})}%
   {\usecounter{mycount}\labelwidth=1cm\itemsep 0pt}}{\end{list}}

\newtheorem{thm}{Theorem}
\newtheorem{lemma}[thm]{Lemma}
\newtheorem{prop}[thm]{Proposition}

\newtheorem{conj}[thm]{Conjecture}

\numberwithin{equation}{section}
\numberwithin{thm}{section}
\numberwithin{figure}{section}

\newcommand\saw{\sigma}
\newcommand\fpq{\phi_{p,q}}
\newcommand\qq{\qquad}
\newcommand\rc{random-cluster}
\newcommand\fpqb{\fpq^b}
\newcommand\fpqo{\fpq^0}
\newcommand\fpqon{\fpq^1}

\newcommand\sD{{\mathcal D}}

\newcommand\pcb{\pc^b}
\newcommand\pco{\pc^0}
\newcommand\pcon{\pc^1}
\newcommand\la{\lambda}
\renewcommand\AA{\mathbb{A}}

\newcommand\td{{\text{\rm d}}}
\renewcommand\o{\text{\rm o}}
\newcommand\dc{d_{\text{\rm c}}}
\newcommand\psd{p_{\text{\rm sd}}}

\newcommand\tp{\text{\rm per}}
\newcommand\tf{0}
\newcommand\tw{1}

\counterwithin{table}{section}

\endlocaldefs

\begin{document}

\begin{frontmatter}

\title{Three theorems in\\ discrete random geometry}

\runtitle{Three theorems in discrete random geometry}

\begin{aug}
\author{\fnms{Geoffrey} \snm{Grimmett}\ead[label=e1]{g.r.grimmett@statslab.cam.ac.uk}}
\address{Centre for Mathematical Sciences\\
University of Cambridge\\
Wilberforce Road\\
Cambridge CB3 0WB, UK\\
\printead{e1}}

\runauthor{G. Grimmett}
\end{aug}

\begin{abstract}
These notes are focused on three recent results in
discrete random geometry, namely: the proof by Duminil-Copin and Smirnov
that the connective constant of the hexagonal lattice is $\sqrt{2+\sqrt 2}$;
the proof by the author and Manolescu
of the universality of inhomogeneous bond percolation on
the square, triangular, and hexagonal lattices;
the proof by Beffara and Duminil-Copin that the critical point of the
random-cluster model on $\ZZ^2$ is $\sqrt q/(1+\sqrt q)$.
Background information
on the relevant random processes is presented on route
to these theorems.
The emphasis is upon the communication of ideas
and connections as well as upon the detailed proofs.
\end{abstract}



\received{\smonth{10} \syear{2011}, corrected by the author 27 January 2012}

\tableofcontents

\end{frontmatter}

\section{Introduction}\label{sec:intro}

These notes are devoted to three recent rigorous results of
significance in the area of discrete random geometry in two dimensions.
These results are concerned with self-avoiding walks, percolation, and the \rc\ model, and may be summarized as:
\begin{letlist}
\item
the connective constant for self-avoiding walks on the hexagonal lattice is $\sqrt{2+\sqrt 2}$,
\cite{Dum-S}.
\item  the universality of inhomogeneous bond percolation on
the square, triangular and hexagonal lattices, \cite{GM2},
\item the critical point of the \rc\ model on the square lattice with cluster-weighting factor $q \ge 1$
is $\sqrt q/(1+\sqrt q)$, \cite{Beffara_Duminil}.
\end{letlist}
In each case, the background and context will be described and the theorem stated.
A complete proof is included in the case of self-avoiding walks,
whereas reasonably detailed outlines are presented in the other two cases.

If the current focus is on three specific
theorems, the general theme is two-dimensional stochastic
systems. In an exciting area of current research initiated by Schramm \cite{Sch00, Sch06},
connections are being forged between
discrete models and conformality; we mention percolation \cite{Smirnov},
the Ising model \cite{Chelkak-Smirnov2}, uniform
spanning trees and loop-erased random walk \cite{SLW04},
the discrete Gaussian free field \cite{SS09}, and self-avoiding walks \cite{DGKLP}.
In each case, a scaling limit leads (or will lead) to a conformal structure characterized by
a Schramm--L\"owner evolution (SLE). In the settings of (a), (b), (c) above,
the relevant scaling limits are yet to be proved, and in that sense this article is
about three \lq pre-conformal' families of stochastic processes.

There are numerous surveys and books covering the history and
basic methodology of these processes, and we do not repeat this material here.
Instead, we present clear definitions of the processes in question, and we outline
those parts of the general theory to be used in the proofs of the above three theorems.
Self-avoiding walks (SAWs) are the subject of
Section~\ref{sec:saw}, bond percolation of Section~\ref{sec:bp}, and
the \rc\ model of Section~\ref{sec:rcm}.   More expository material
about these three topics may be found, for example,  in \cite{Grimmett_Graphs},
as well as:
SAWs \cite{MS}; percolation \cite{BolRio,Grimmett_Percolation,WW_park_city}; the \rc\ model
\cite{Grimmett_RCM,Werner_SMF}. The relationship between SAWs, percolation, and SLE is
sketched in the companion paper \cite{Law11}. Full references to original material
are not invariably included.

A balance is attempted in these notes between providing enough but not too much
basic methodology. One recurring topic that might
delay readers is the theory of stochastic inequalities. Since a sample
space of the form $\Om=\{0,1\}^E$ is  a partially
ordered set, one may speak of \emph{increasing} random variables. This in turn
gives rise to a partial order on probability measures\footnote{The expectation of a random variable $X$ under a probability
measure $\mu$ is written $\mu(X)$.} on $\Om$ by: $\mu \lest \mu'$ if
$\mu(X) \le \mu'(X)$ for all increasing random variables $X$.
Holley's theorem \cite{Hol}
provides a useful sufficient
criterion for such an inequality in the context of this article.
The reader is referred to \cite[Chap. 2]{Grimmett_RCM} and \cite[Chap. 4]{Grimmett_Graphs}
for accounts of Holley's
theorem, as well as of `positive association' and the FKG inequality.

\begin{figure}[t]
\centering
    \includegraphics[width=0.9\textwidth]{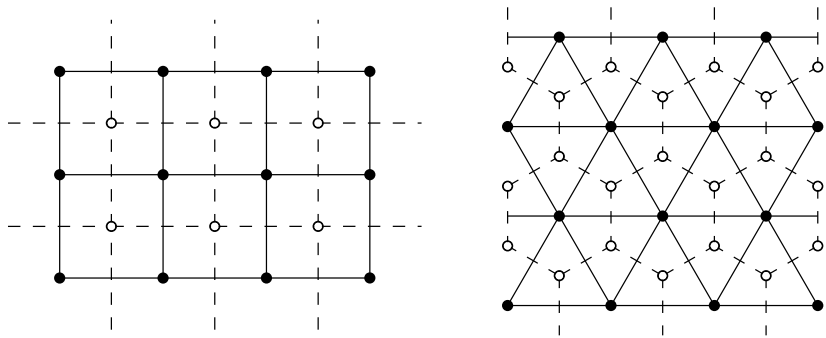}
  \caption{The square lattice $\LL^2$ and its dual square lattice. The triangular lattice $\TT$
and its dual hexagonal
(or `honeycomb') lattice $\HH$.}
  \label{fig:lattices}
\end{figure}

A variety of lattices will be encountered in this article, but predominantly
the square, triangular, and hexagonal lattices illustrated in Figure~\ref{fig:lattices}.
More generally, a \emph{lattice} in $d$ dimensions is a connected
graph $\sL$ with bounded vertex-degrees, together with an embedding
in $\RR^d$ such that: the embedded graph is locally-finite and invariant under shifts
of $\RR^d$ by any of $d$ independent vectors $\tau_1, \tau_2, \dots,\tau_d$.
We shall sometimes speak of a lattice without having regard to its embedding. A lattice is
\emph{vertex-transitive} if, for any pair $v$, $w$ of vertices, there exists a graph-automorphism
of $\sL$ mapping $v$ to $w$. For simplicity, we shall consider only vertex-transitive lattices.
We pick some vertex of a lattice $\sL$ and designate it the
\emph{origin}, denoted $0$, and we generally
assume that $0$ is embedded at the origin of $\RR^d$.
The \emph{degree} of a vertex-transitive
lattice is the number of edges incident to any given vertex.
We write $\LL^d$ for the $d$-dimensional (hyper)cubic lattice, and $\TT$, $\HH$
for the triangular and hexagonal lattices.

\section{Self-avoiding walks}\label{sec:saw}

\subsection{Background}

Let $\sL$ be a  lattice with origin  $0$, and assume for simplicity that
$\sL$ is vertex-transitive.
A {\it self-avoiding walk\/} (SAW) is a lattice path that visits
no vertex more than once.

How many self-avoiding walks of length $n$ exist, starting from the origin?
What is the `shape' of such a SAW chosen at random? In particular,
what can be said about
the distance between its endpoints? These and related questions
have attracted a great deal of attention since the notable
paper \cite{HM} of Hammersley and Morton, and never more so than in recent years.
Mathematicians believe but have not proved that a
typical SAW on a two-dimensional lattice  $\sL$, starting
at the origin, converges in
a suitable manner as $n\to\oo$ to a SLE$_{8/3}$ curve.
See \cite{DGKLP, MS, Sch06, Smi07}
for discussion and results.

Paper \cite{HM} contained a number of stimulating
ideas, of which we mention here the use of subadditivity in studying asymptotics.
This method and its elaborations have proved extremely fruitful in
many contexts since. Let $\sS_n$ be the set of
SAWs with length $n$ starting at the origin, with cardinality $\saw_n=
\saw_n(\sL) :=|\sS_n|$.

\begin{lemma}[\cite{HM}]\label{submult}
We have that $\saw_{m+n} \le \saw_m \saw_n$, for $m,n\ge 0$.
\end{lemma}

\begin{proof}
Let $\pi$ and $\pi'$ be finite SAWs starting at the origin, and denote
by $\pi \ast \pi'$ the walk obtained by following $\pi$ from 0 to
its other endpoint $x$, and then following the translated walk $\pi'+x$.
Every $\nu \in \sS_{m+n}$ may be written in a unique way as $\nu=\pi\ast\pi'$
for some $\pi\in\sS_m$ and $\pi'\in\sS_n$. The claim of the lemma follows.
\end{proof}

\begin{thm}\label{conncon}
Let $\sL$ be a vertex-transitive lattice in $d\ge 2$ dimensions with degree $\De$.
The limit $\kappa = \kappa(\sL) =\lim_{n\to\oo} (\saw_n)^{1/n}$ exists and satisfies
$1 < \kappa \le \De-1$.
\end{thm}

\begin{proof}
By Lemma~\ref{submult}, $x_m=\log\saw_m$ satisfies
the `subadditive inequality'
\begin{equation*}\label{subineq}
x_{m+n} \le x_m +x_n.
\end{equation*}
By the subadditive-inequality theorem\footnote{Sometimes known as Fekete's Lemma.}
(see \cite[App. I]{Grimmett_Percolation}),
the limit
$$
\lambda=\lim_{n\to\oo} \frac 1n x_n
$$
exists, and we write $\kappa= e^\la$.

Since there are at most $\De-1$ choices for each step of a SAW (apart from the first),
we have that $\saw_n \le \De(\De-1)^{n-1}$, giving that $\kappa \le \De-1$.
It is left as an exercise to show $\kappa>1$.
\end{proof}

The constant $\kappa = \kappa(\sL)$ is called the
{\it connective constant\/} of the lattice $\sL$. The exact value of
$\kappa=\kappa(\LL^d)$ is unknown for every $d\ge 2$, see \cite[Sect.\ 7.2, pp.\ 481--483]{BDH}.
As explained in the next section, the hexagonal lattice has a special structure which
permits an exact calculation, and our main purpose here is to present the proof of this.
In addition, we include next
a short discussion of critical exponents, for a fuller discussion of which the reader is
referred to \cite{BDGS,MS,slade10}

By Theorem~\ref{conncon}, $\saw_n$ grows exponentially
in the manner of  $\kappa^{n(1+\o(1))}$.
It is believed by mathematicians and physicists
that there is a power-order correction, in that
$$
\saw_n \sim A_1 n^{\g-1} \kappa^n,
$$
where the exponent $\g$ depends only on the number of dimensions and not otherwise
on the lattice (here and later, there should be an additional logarithmic correction in four dimensions).
Furthermore, it is believed that $\g=\frac {43}{32}$ in two dimensions.
We mention two further critical exponents. It is believed that
the (random) end-to-end distance $D_n$ of a typical $n$-step SAW satisfies
$$
E(D_n^2) \sim A_2n^{2\nu},
$$
and furthermore $\nu=\frac 34$ in two dimensions. Finally, let
$Z_v(x) = \sum_{k} \si_k(v)x^k$ where $\si_k(v)$ is the number of $k$-step SAWs
from the origin to the vertex $v$. The generating function $Z_v$ has
radius of convergence $\kappa^{-1}$, and it is believed that
$$
Z_v(\kappa^{-1}) \sim A_3 |v|^{-(d-2+\eta)}\qq \text{as } |v| \to \oo,
$$
in $d$ dimensions. Furthermore, $\eta$ should satisfy the so-called Fisher relation
$\g=(2-\eta)\nu$, giving that $\eta=\frac 5{24}$ in two dimensions.
The numerical predictions for these exponents in two dimensions
are explicable on the conjectural basis that a typical $n$-step SAW
in two dimensions converges to SLE$_{8/3}$ as $n\to\oo$
(see \cite{Beffara2010,SLW04b}).

\subsection{Hexagonal-lattice connective constant}

\begin{thm}[\cite{Dum-S}]\label{hex-conn}
The connective constant of the hexagonal lattice $\HH$ satisfies $\kappa(\HH)=
\sqrt{2+\sqrt 2}$.
\end{thm}

This result of Duminil-Copin and Smirnov provides a rigorous and provocative
verification of a prediction of Nienhuis \cite{Nienhuis} based on
conformal field theory. The proof falls short of a proof of
conformal invariance for self-avoiding walks on $\HH$. The remainder
of this section contains an outline of the proof of Theorem~\ref{hex-conn},
and is drawn from \cite{Dum-S}\footnote{The
proof has been re-worked in \cite{K11}.}.

The reader may wonder about the special nature of the hexagonal lattice. It is something of a mystery
why certain results for this lattice (for example, Theorem~\ref{hex-conn}, and the conformal
scaling limit of `face' percolation) do not yet extend to other lattices.

\begin{figure}[t]
 \centering
   \includegraphics{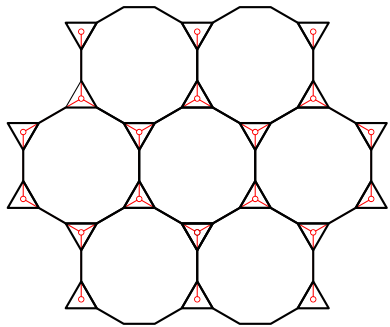}
  \caption{The Archimedean lattice $(3,12^2)$ is obtained by replacing each vertex
of the hexagonal lattice $\HH$ by a triangle.}
  \label{fig:312}
\end{figure}

We note an application of Theorem~\ref{hex-conn} to the  lattice illustrated in
Figure~\ref{fig:312}, namely the Archimedean lattice denoted $\AA = (3,12^2)$
and known also as a `Fisher lattice' after \cite{MF66}. Edges of $\AA$ lying in a triangle
are called \emph{triangular}. For simplicity, we shall consider
only SAWs of $\AA$ that start with a triangular edge
and finish with a non-triangular edge.
The non-triangular edges of such a SAW $a$ induce
a SAW $h$ of $\HH$. Furthermore, for given $h$, the corresponding $a$ are
obtained by replacing each vertex $v$ of $h$ (other
than its final vertex) by either a single edge in the
triangle of $\AA$ at $v$, or by two such edges.
It follows that
the generating function of such walks is
$$
Z^\AA(a) =\sum_k\si_k(\HH)(a^2+a^3)^k = Z(a^2+a^3),
$$
where $Z(x) := \sum_n \si_n(\HH)x^k$.
The radius of convergence of $Z^\AA$ is $1/\kappa(\AA)$, and we deduce the
following formula of \cite{JenGutt}:
\begin{equation}
\frac 1 {\kappa(\AA)^2} + \frac 1 {\kappa(\AA)^3}
=\frac 1{\kappa(\HH)}.\label{connkag}
\end{equation}

One may show similarly that the critical exponents $\gamma$, $\nu$, $\eta$ are equal
for $\AA$ and $\HH$, assuming they satisfy suitable definitions. The details
are omitted.

\begin{proof}[Proof of Theorem~\ref{hex-conn}]
This exploits the relationship between  $\RR^2$ and the Argand diagram
of the complex numbers $\CC$. We embed $\HH=(V,E)$ in $\RR^2$ in a natural way:
edges have  length 1 and are inclined at angles  $\pi/6$, $\pi/2$, $5\pi/6$ to the
$x$-axis, the origins of $\sL$ and $\RR^2$ coincide,
and the line-segment from $(0,0)$ to $(0,1)$ is an edge of $\HH$. Any point in $\HH$ may
thus be represented by a complex number.
Let $\sM$ be the set of  midpoints of edges of $\HH$.
Rather than counting paths between \emph{vertices} of $\HH$, we count paths between
\emph{midpoints}.

Fix $a \in \sM$, and let
$$
Z(x) = \sum_{\gamma} x^{|\gamma|}, \qq x \in (0,\oo),
$$
where the sum is over all SAWs $\gamma$ starting at $a$, and $|\gamma|$ is the number
of vertices
visited by $\gamma$. Theorem~\ref{hex-conn} is equivalent to the assertion that
the radius of convergence of $Z$ is $\chi := 1/\sqrt{2+\sqrt 2}$.
We shall therefore prove that
\begin{align}
Z(\chi) &= \oo,\label{Z1}\\
Z(x) &< \oo \quad\text{ for } x < \chi. \label{Z2}
\end{align}

Towards this end we  introduce a  function that records the turning angle of a SAW.
A SAW $\g$ departs from its starting-midpoint $a$ in one of two possible directions,
and its direction changes by $\pm \pi/3$ at each vertex. On arriving at
its other endpoint $b$,
it has turned through some total \emph{turning
angle} $T(\g)$, measured anticlockwise in radians.

We work within some bounded region $M$ of $\HH$.
Let $S \subseteq V$ be a finite set of vertices that induces a connected subgraph,
and let $M=M_S$ be the
set of midpoints of edges touching points in $S$. Let $\De M$ be the set of midpoints for
which the corresponding edge of $\HH$
has exactly one endpoint in $S$. Later in the proof we shall restrict $M$ to a region of the type
illustrated in Figure~\ref{fig:omega}.

\begin{figure}[t]
 \centering
   \includegraphics[width=0.7\textwidth]{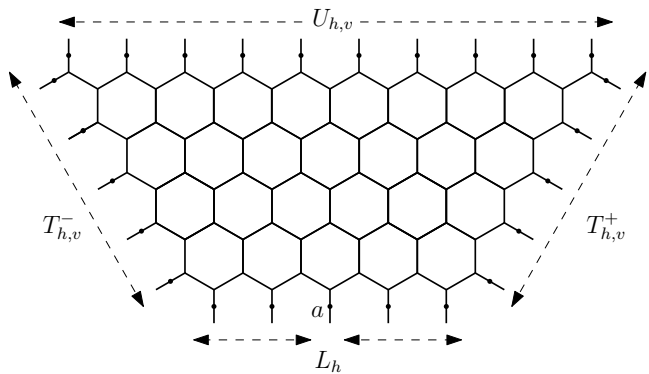}
  \caption{The region $M_{h,v}$ has $2h+1$ midpoints on the bottom side,
and $v$ at the left/right sides. In this illustration, we have $h=2$ and $v=5$.}
  \label{fig:omega}
\end{figure}

Let $a \in \De M$ and $\si,x \in(0,\oo)$,
and define the so-called `parafermionic observable' of \cite{Dum-S} by
\begin{equation}\label{paraf}
F_{a,M}^{\si,x}(z) = \sum_{\gamma:a \to z} e^{-i\si T (\g)}x^{|\g|}, \qq z \in M,
\end{equation}
where the summation is over all SAWs from $a$ to $z$ lying entirely in $M$.
We shall suppress some of the notation in  $F_{a,M}^{\si,x}$ when no ambiguity ensues.
The key ingredient of the proof is the following lemma,
which is strongly suggestive of discrete analyticity.

\begin{lemma}\label{hex-mainlem}
Let $\si=\frac 58$ and $x= \chi$. For $v \in S$,
\begin{equation}\label{hex-crit}
(p-v)F(p) + (q-v)F(q) + (r-v)F(r) = 0,
\end{equation}
where $p, q, r \in \sM$ are the midpoints of the three edges incident to $v$.
\end{lemma}

The quantities in \eqref{hex-crit} are to be interpreted as complex numbers.

\begin{proof}[Proof of Lemma~\ref{hex-mainlem}]
Let $v \in S$. We assume for definiteness that the star at $v$ is as drawn
on the left of Figure~\ref{fig:vpqr0}\footnote{This and later figures are viewed best
in colour.}.

\begin{figure}[t]
 \centering
   \includegraphics[width=0.8\textwidth]{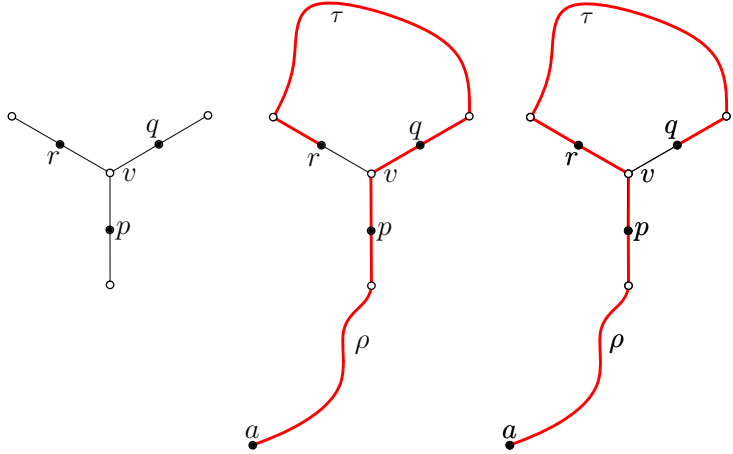}
  \caption{The star centred at the vertex $v$ is on the left.
Two SAWs lying in $\sP_3$, visiting the midpoints of edges in the respective orders $pqr$ and $prq$.
They follow the same SAWs $\rho$ and $\tau$ (in one or the other directions), and differ
only within the star.}
  \label{fig:vpqr0}
\end{figure}

Let $\sP_k$ be the set of SAWs of $M$ starting at $a$  whose intersection
with $\{p,q,r\}$ has cardinality $k$, for $k=1,2,3$.
We shall show that the aggregate contribution to \eqref{hex-crit}
of $\sP_1\cup \sP_2$ is zero, and similarly of $\sP_3$.

Consider first $\sP_3$. Let $\g \in \sP_3$, and write $b_1$,
$b_2$, $b_3$ for the  ordering of $\{p,q,r\}$ encountered along $\g$ starting at $a$.
Thus $\g$ comprises:
\begin{numlist}
\item[--] a SAW $\rho$ from $a$ to $b_1$,
\item[--] a SAW of length 1 from $b_1$ to $b_2$,
\item[--] a SAW  $\tau$ from $b_2$ to $b_3$ that is disjoint from $\rho$,
\end{numlist}
as illustrated in Figure~\ref{fig:vpqr0}.
We partition $\sP_3$ according to the pair $\rho$, $\tau$. For given $\rho$, $\tau$,
the aggregate contribution of these two paths to the left side of \eqref{hex-crit} is
\begin{equation}\label{g3}
c \left( \ol \th e^{-i\si 4\pi/3} + \th e^{i\si4\pi/3  }  \right)
\end{equation}
where $c=(b_1-v) e^{-i\si T(\rho)}x^{|\rho|+|\tau|+1}$ and
$$
\th=\frac{q-v}{p-v} = e^{i2\pi/3}.
$$
The parenthesis in \eqref{g3} equals $2\cos\bigl(\frac23\pi(2\si+1)\bigr)$ which
is $0$ when $\si=\frac 58$.

\begin{figure}[t]
 \centering
   \includegraphics{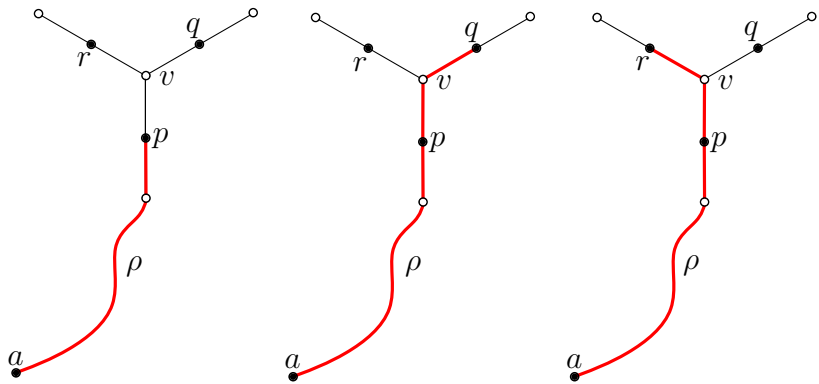}
  \caption{The left SAW, denoted $\g_p$, intersects the set $\{p,q,r\}$ once only.
The other two paths, denoted $\g_q$ and $\g_r$,  are extensions of the first.}
  \label{fig:vpqr}
\end{figure}

Consider now $\sP_1 \cup \sP_2$. This set may be
partitioned according to the point  $b$ in $\{p,q,r\}$ visited first, and by the route $\rho$ of the SAW from
$a$ to $b$. For given $b$, $\rho$, there are exactly three such
SAWs, as in Figure~\ref{fig:vpqr}.
Their aggregate contribution to
the left side of \eqref{hex-crit} is
$$
c\left(1+x\th e^{i\si\pi/3} + x\ol \th e^{-i\si\pi/3} \right)
$$
where $c= (b-v)e^{-i\si T(\g_b)}x^{|\g_b|}$.
With $\si=\frac58$, we set this to $0$ and solve for $x$, to find
$x=1/[2\cos(\pi/8)] = \chi$. The lemma is proved.
\end{proof}

We return to the proof of Theorem~\ref{hex-conn}, and we set $\si=\frac58$ henceforth.
Let $M=M_{h,v}$ be as in Figure~\ref{fig:omega}, and let
$L_h$, $T^\pm_{h,v}$, $U_{h,v}$ be the sets of midpoints indicated in
the figure (note that $a$ is excluded from $L_h$). Let
$$
\la^x_{h,v} = \sum_{\g:a \to L_{h}} x^{|\g|},
$$
where the sum is over all SAWs in $M_{h,v}$ from $a$ to some point  in $L_{h}$.
All such $\g$ have $T(\g) = \pm\pi$.
The sums $\tau^{\pm,x}_{h,v}$ and $\ups_{h,v}^x$ are defined similarly in terms of SAWs ending in
$T^\pm_{h,v}$ and $U_{h,v}$ respectively, and all such $\g$ have $T(\g) = \mp 2\pi/3$ and $T(\g) = 0$
\resp.

In summing \eqref{hex-crit} over all vertices $v$ of $M_{h,v}$, with $x=\chi$,
all contributions cancel except those from
the boundary midpoints. Using the symmetry of $M_{h,v}$,
we deduce that
$$
-iF^\chi(a) - i\Re(e^{i\si\pi}) \la^\chi_{h,v} +
i\th e^{-i\si2\pi/3}\tau^{-,\chi}_{h,v} + i\ups^\chi_{h,v} +
i\ol{\th} e^{i\si 2\pi/3} \tau^{+,\chi}_{h,v}=0.
$$
Divide by $i$,  and use
the fact that $F^\chi(a) = 1$, to obtain
\begin{equation}\label{g4}
c_l\la^\chi_{h,v} + c_t\tau^\chi_{h,v} + \ups^\chi_{h,v} = 1,
\end{equation}
where $\tau_{h,v}= \tau^+_{h,v} + \tau^-_{h,v}$,  $c_l = \cos(3\pi/8)$, and $c_t = \cos(\pi/4)$.

Let $x \in (0,\oo)$.
Since $\la^x_{h,v}$ and $\ups^x_{h,v}$ are increasing
in $h$, the limits
$$
\la^x_v = \lim_{h\to\oo} \la^x_{h,v},\quad
\ups^x_v = \lim_{h\to\oo} \ups^x_{h,v},
$$
exist. Hence, by \eqref{g4}, the decreasing limit
\begin{equation}\label{g7}
\tau^\chi_{h,v} \downarrow \tau^\chi_v\qquad \text{as } h \to \oo,
\end{equation}
exists also. Furthermore, by \eqref{g4},
\begin{equation}\label{g5}
c_l\la^\chi_v + c_t\tau^\chi_v + \ups^\chi_v = 1.
\end{equation}
We shall use \eqref{g7}--\eqref{g5} to prove \eqref{Z1}--\eqref{Z2} as follows.

\noindent
\emph{Proof of \eqref{Z1}.} There are two cases depending on whether or not
\begin{equation}\label{g6}
\tau_v^\chi>0 \qq \text{for some } v \ge 1.
\end{equation}
Assume first that \eqref{g6} holds\footnote{In fact, \eqref{g6} does not hold, see \cite{BGG}.},
and pick $v\ge 1$ accordingly.
By \eqref{g7}, $\tau^\chi_{h,v} \ge \tau_v^\chi$
for all $h$, so that
$$
Z(\chi) \ge \sum_{h=1}^\oo \tau^\chi_{h,v} = \oo,
$$
and \eqref{Z1} follows.

\begin{figure}[t]
 \centering
   \includegraphics[width=0.7\textwidth]{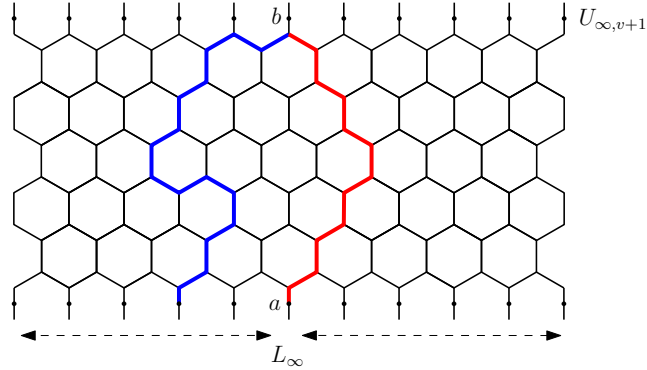}
  \caption{A SAW contributing to $\la^\chi_{v+1}$ but not $\la^\chi_v$
is broken at its first highest vertex into two SAWs coloured red and blue.
By adding two half-edges, we obtain two SAWs from $b\in U_{\oo,v+1}$ to $L_{\oo}$.}
  \label{fig:bkl+}
\end{figure}

Assume now that \eqref{g6} is false so that, by \eqref{g5},
\begin{equation}\label{g8}
c_l\la^\chi_v + \ups^\chi_v = 1,\qq  v \ge 1.
\end{equation}
We propose to bound $Z(\chi)$ below in terms of the $\ups^\chi_v$. The
difference $\la^\chi_{v+1} - \la^\chi_v$ is the sum
of $\chi^{|\g|}$ over
all $\g$ from $a$ to $L_{\oo}$ whose highest vertex lies between $U_{\oo,v}$
and $U_{\oo,v+1}$. See Figure~\ref{fig:bkl+}.
We split such a $\g$ into two pieces
at its first highest vertex, and add two half-edges
to obtain two self-avoiding paths from a given midpoint, $b$ say, of $U_{\oo,v+1}$ to $L_{\oo}\cup\{a\}$.
Therefore,
$$
\la^\chi_{v+1} - \la^\chi_v \le \chi(\ups^\chi_{v+1})^2, \qq v \ge 1.
$$
By \eqref{g8},
$$
c_l\chi(\ups^\chi_{v+1})^2 + \ups^\chi_{v+1} \ge \ups^\chi_v, \qq v \ge 1,
$$
whence, by induction,
$$
\ups^\chi_v \ge \frac 1v \min \left\{\ups^\chi_1,\frac1{c_l\chi }\right\}, \qq v \ge 1.
$$
Therefore,
$$
Z(\chi) \ge \sum_{v=1}^\oo \ups^\chi_v = \oo.
$$

\noindent
\emph{Proof of \eqref{Z2}.} SAWs that start at a lowermost vertex and end at an uppermost
vertex (or \emph{vice versa}) are called \emph{bridges}.
Hammersley and Welsh \cite{HW62} showed, as follows, that
any SAW may be decomposed in a unique way into sub-walks that are bridges.
They used this to obtain a bound on the rate of convergence in the limit
defining the connective constant.

\begin{figure}[t]
 \centering
   \includegraphics[width=0.6\textwidth]{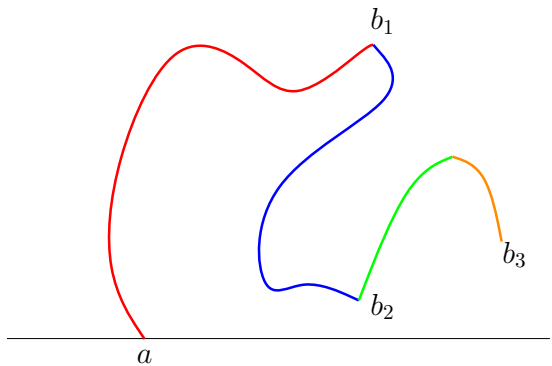}
  \caption{A half-plane walk $\g$ decomposed into four bridges.}
  \label{fig:bridge}
\end{figure}

Consider first a
(finite) SAW $\g$ in $M_{\oo,\oo}$  starting at $a$. From amongst
its highest vertices, choose the last, $b_1$ say.
Consider the sub-walk from $b_1$ onwards, and find the final lowest vertex, $b_2$ say.
Iterate the procedure until the endpoint of $\g$ is reached,
as illustrated in Figure~\ref{fig:bridge}. The outcome is a decomposition of $\g$
into an ordered set of bridges with vertical displacements written $T_0 > T_1 > \cdots > T_j$.

Now let $\g$ be a SAW from $a$ (not necessarily a half-plane walk).
Find the earliest vertex of $\g$ that is lowest, $c$ say. Then $\g$ may be decomposed into
a walk $\g_1$ from $a$ to $c$, together with the remaining walk $\g_2 = \g\setminus \g_1$.
On applying the above procedure to $\g_1$ viewed backwards, we obtain
a decomposition into bridges with vertical displacements written $T_{-i} < \cdots <T_{-1}$,
Similarly, $\g_2$ has a decomposition with  $T_0 > \cdots > T_j$. The original walk $\g$ may
be reconstructed from knowledge of the constituent bridges. A little further care
is needed in our case since our walks connect midpoints rather than vertices.

One may deduce a bound for $Z(x)$ in terms of the $\ups_v^x$.
Any SAW from $a$ has two choices
for initial direction, and thereafter has a bridge decomposition as above. Therefore,
\begin{equation}\label{G40}
Z(x) \le 2\sum_{\substack{T_{-i} < \cdots < T_{-1}\\ T_0 > \cdots > T_j}}
\,\prod_{k=-i}^j \ups_{T_k}^x
= 2\prod_{T=1}^\oo (1+\ups^x_T)^2.
\end{equation}
It remains to bound the right side.

Since all SAWs from $a$ to $U_{h,v}$ have length at least $h$,
$$
\ups^x_{h} \le \left(\frac x \chi\right)^h \ups^\chi_{h} \le \left(\frac x \chi\right)^h,
\qq x\le\chi.
$$
Therefore,
$$
\prod_{T=1}^\oo (1+\ups^x_T) < \oo, \qq x < \chi,
$$
and \eqref{Z2} follows by \eqref{G40}.
\end{proof}

\section{Bond percolation}\label{sec:bp}

\subsection{Background}

Percolation is the fundamental stochastic model for spatial disorder.
We consider bond percolation on several lattices including
the (two-dimensional) square, triangular
and hexagonal lattices of Figure~\ref{fig:lattices},
and the (hyper)cubic
lattices $\LL^d=(\ZZ^d,\EE^d)$ in $d \ge 3$ dimensions. Detailed accounts of the basic theory
may be found in \cite{Grimmett_Percolation, Grimmett_Graphs}.

Percolation comes in two forms, `bond' and `site', and we concentrate here
on the bond model. Let $\sL=(V,E)$ be  a lattice with origin denoted 0,
and let $p\in[0,1]$. Each edge $e \in E$ is designated
either {\it open\/} with probability $p$, or {\it closed\/} otherwise,
different edges receiving independent states. We think of an open
edge as being open to the passage of some material such as disease, liquid,
or infection. Suppose we remove
all closed edges, and consider the remaining open subgraph of the lattice.
Percolation theory is concerned with the geometry of this open
graph. Of special interest is the size and shape of
the open cluster $C_0$ containing the origin, and in particular the
probability that $C_0$ is infinite.

The sample space is the set $\Om=\{0,1\}^{E}$ of $0/1$-vectors $\om$ indexed
by the edge-set $E$; here, 1 represents `open', and 0 `closed'. The
probability measure is product measure $\PP_p$ with density $p$.

For $x,y\in V$, we write $x \lra y$ if there exists an open path
joining $x$ and $y$. The {\it open cluster\/} at $x$ is the
set $C_x =\{y: x\lra y\}$ of all
vertices reached along open paths from the vertex $x$,
and we write $C=C_0$.
The {\it percolation probability\/} is the function
$\theta(p)$ given by
$$
\theta(p) = \PP_p(|C|=\oo),
$$
and the \emph{critical probability} is defined by
\begin{equation}\label{defcritprob}
\pc = \pc(\sL) = \sup\{p: \theta(p)=0\}.
\end{equation}
It is elementary that $\theta$ is a
non-decreasing function, and therefore,
$$
\theta(p)\, \begin{cases} = 0 &\text{if } p<\pc,\\
>0 &\text{if } p>\pc.
\end{cases}
$$
It is a fundamental fact that $0<\pc(\sL)<1$ for any lattice
$\sL$ in two or more dimensions, but it is unproven in general that
no infinite open cluster exists when $p=\pc$.

\begin{conj}\label{thetapc0}
For any lattice $\sL$ in $d \ge 2$ dimensions, we have that
$\theta(\pc) = 0$.
\end{conj}

The claim of the conjecture is known to be valid for certain lattices when $d=2$
and for large $d$, currently $d \ge 19$.

Whereas the above process is defined in terms
of a single parameter $p$,
much of this section is directed at the multiparameter setting in which
an edge $e$ is designated open with some probability $p_e$.
In such a case, the critical probability $\pc$ is replaced by a so-called `critical surface'.
See Section~\ref{sec:univ} for a more precise discussion of this.

The theory of percolation is extensive and influential. Not only is percolation a
benchmark model for studying random spatial processes in general, but also it has been,
and continues to be, a source of beautiful open problems (of which Conjecture~\ref{thetapc0} is one).
Percolation in two dimensions has been especially prominent in the last decade by virtue
of its connections to conformal invariance and conformal field theory. Interested readers
are referred to the papers \cite{Cardy,Sch06,Smirnov, Smi07, Sun11, WW_park_city} and the books
\cite{BolRio, Grimmett_Percolation,Grimmett_Graphs}.

The concepts of critical exponent and scaling are discussed in Section
\ref{sec:pt}. Section~\ref{sec:bxp} is concerned with percolation in two dimensions,
and especially the \bxp. The \stt\ features in Section~\ref{sec:stt}, followed
by a discussion of universality in Section~\ref{sec:univ}.

\subsection{Power-law singularity}\label{sec:pt}

Macroscopic functions, such as the percolation probability and mean cluster-size,
$$
\theta(p)=\PP_p(|C|=\oo),\quad \chi(p)=\PP_p(|C|),
$$
have  singularities at $p=\pc$, and there is overwhelming theoretical and numerical evidence
that these are of `power-law' type. A great deal of effort has been directed towards
understanding the nature of the percolation
phase transition. The picture is now fairly clear for one specific model
in $2$ dimensions (site percolation on the triangular lattice),
owing to
the very significant progress in recent years linking critical percolation
to the Schramm--L\"owner curve SLE$_6$. There remain however
substantial difficulties to be overcome even when $d=2$,
associated largely with the extension of such results to general two-dimensional systems.
The case of large $d$ (currently, $d \ge 19$) is also
well understood, through work based on the so-called `lace expansion'.
Many problems remain open in the obvious case $d=3$.

The nature of the percolation singularity is
expected to be canonical, in that it shares general
features with phase transitions of other models of statistical mechanics.
These features are sometimes referred to as `scaling theory' and they
relate to the `critical exponents' occurring in the power-law
singularities (see \cite[Chap. 9]{Grimmett_Percolation}).
There are two sets of critical exponents, arising firstly in the limit as
$p\to\pc$, and secondly in the limit over increasing
spatial scales when $p=\pc$. The definitions of the critical exponents
are found  in Table~\ref{Tab-ce} (taken from \cite{Grimmett_Percolation}).

\begin{table}[t]
\centering
\tabcolsep=7.5pt
\caption{Eight functions and their critical exponents. The first five exponents arise
in the limit as $p \to \pc$, and the remaining three as $n\to\oo$ with $p=\pc$.
See \cite{Grimmett_Percolation} for a definition
 of the correlation length $\xi(p)$}
\label{Tab-ce}
\begin{tabular}{|c|c|c|c|}
\hline
\multicolumn{2}{|c|}{}&&\\
\multicolumn{2}{|c|}{\emph{Function}}  & \emph{Behaviour} & \emph{Exp.} \\
\multicolumn{2}{|c|}{}&&\\
\hline
&&&\\
\raisebox{1ex}{percolation}& $\th (p)=\PP_p(|C|=\infty )$ & $\th (p)\approx (p-\pc )^\beta$ & $\beta$ \\
\raisebox{.8ex}{probability}&&&\\
&&&\\
\raisebox{1ex}{truncated}& $\chi^{\text f}(p)=\PP_p(|C|1_{|C|<\infty})$& $\chi^{\text f}(p)\approx |p-\pc |^{-\gamma}$ & $\gamma$\\
\raisebox{.8ex}{mean cluster-size} &&& \\
&&&\\
\raisebox{1ex}{number of}& $\kappa (p)=\PP_p(|C|^{-1})$ & $\kappa'''(p)\approx |p-\pc |^{-1-\alpha}$ & $\alpha$\\
\raisebox{.8ex}{clusters per vertex}&&&\\
&&&\\
cluster moments& $\chi_k^{\text f} (p)=\PP_p(|C|^k1_{|C|<\infty} )$ & $\displaystyle\frac{\chi_{k+1}^{\text f}
     (p)}{\chi_k^{\text f}(p)}\approx |p-\pc |^{-\De}$ & $\De$\\
&&&\\
correlation length& $\xi (p)$ & $\xi (p)\approx |p-\pc |^{-\nu}$ & $\nu$\\
&&&\\
\hline
\multicolumn{2}{|c|}{}&&\\
\multicolumn{2}{|c|}{cluster volume}& $\PP_{\pc} (|C|=n)\approx n^{-1-1/\de}$ & $\de$\\
\multicolumn{2}{|c|}{}&&\\
\multicolumn{2}{|c|}{cluster radius} & $\PP_{\pc}\bigl(\rad (C)=n\bigr)\approx n^{-1-1/\rho}$ & $\rho$\\
\multicolumn{2}{|c|}{}&&\\
\multicolumn{2}{|c|}{connectivity function} & $\PP_{\pc}(0\lra x)\approx   \| x\|^{2-d-\eta}$ & $\eta$\\
\multicolumn{2}{|c|}{}&&\\
\hline
\end{tabular}
\end{table}

The notation of Table~\ref{Tab-ce} is as follows.
We write $f(x) \approx g(x)$ as $x \to x_0 \in[0,\oo]$ if
$\log f(x)/\log g(x) \to 1$.
 The \emph{radius} of the open cluster $C_x$ at the vertex $x$ is defined by
$$
\rad (C_x)
=\sup\{\|y\|: x \lra y\},
$$
where
$$
\|y\|= \sup_i |y_i|, \qquad y=(y_1,y_2,\dots,y_d)\in\RR^d,
$$
is the supremum ($L^\oo$) norm on $\RR^d$.
(The choice of norm is irrelevant since all norms are equivalent on $\RR^d$.)
The limit as $p\to\pc$ should be interpreted
in a manner appropriate for the function in question
(for example, as $p \downarrow \pc$ for $\th(p)$,
but as $p\to\pc$ for $\kappa(p)$).
The \emph{indicator function} of an event $A$ is denoted $1_A$.

Eight critical exponents are listed in Table~\ref{Tab-ce},
denoted $\alpha$, $\beta$, $\gamma$, $\de$, $\nu$, $\eta$, $\rho$, $\De$,
but there is no general proof of the
existence of any of these exponents for arbitrary $d\ge 2$.
Such critical exponents may be defined for phase
transitions in a large family of physical systems. However, it is
not believed that they are independent variables, but rather that,
for all such systems, they
satisfy the so-called {\it scaling relations\/}
\begin{gather*}
2-\alpha=\gamma +2\beta =\beta (\de +1),\\
\De=\de\beta ,\quad
\gamma=\nu (2-\eta ),
\end{gather*}
and, when $d$ is not too large, the {\it hyperscaling relations\/}\label{ind-hypers}
\begin{align*}
d\rho=\de +1,\quad
2-\alpha=d\nu .
\end{align*}
More generally, a `scaling relation' is any equation involving critical
exponents believed to be `universally' valid.
The {\it upper critical dimension\/} is the largest value $\dc$
such that the hyperscaling relations hold for $d\leq\dc$. It is believed
that $\dc =6$ for percolation.
There is no general proof of the validity of the scaling and hyperscaling
relations for percolation, although quite a lot is known when either $d=2$ or $d$ is large.
The case of large $d$ is studied via the lace expansion, and this
is expected to be valid for $d > 6$.

We note some further points in the context of percolation.
\begin{letlist}
\item
{\it Universality\/}. The numerical values of critical
exponents are believed to depend only on the value of $d$, and to be independent of the
choice of lattice, and whether bond or site. Universality in two dimensions
is discussed further in Section~\ref{sec:univ}.

\item
{\it Two dimensions\/}. When $d=2$, it is believed that
$$
\a =-\tfrac23,\ \b =\tfrac5{36},\ \gamma =\tfrac{43}{18},\ \de =\tfrac{91}5
,\ldots
$$
These values (other than $\a$) have been proved  (essentially only) in the special case of site
percolation on the triangular lattice, see \cite{Smirnov-Werner}.

\item
{\it Large dimensions\/}. When $d$ is sufficiently large (in fact,
$d\geq\dc$) it is believed that the critical exponents are the same as
those for percolation on a tree (the `mean-field model'), namely $\de =2$,
$\gamma = 1$, $\nu=\frac12$,
$\rho =\frac12$, and so on (the other exponents are
found to satisfy the scaling relations). Using the
first hyperscaling relation, this is consistent with the contention that $\dc =6$.
Several such statements are known to hold for $d\ge 19$, see \cite{HS, HS94, KN}.
\end{letlist}

Open challenges include the following:
\begin{numlist}
\item[--] prove the existence of critical exponents for general lattices,
\item[--] prove some version of universality,
\item[--] prove the scaling and hyperscaling relations in general dimensions,
\item[--] calculate the critical exponents for general models in two dimensions,
\item[--] prove the mean-field values of critical exponents when $d \ge 6$.
\end{numlist}
Progress towards these goals has been substantial. As noted above,
for sufficiently large $d$, the lace expansion has enabled proofs
of exact values for many exponents. There has been remarkable progress
in recent years when $d=2$, inspired
largely by work of Schramm \cite{Sch00}, enacted by Smirnov \cite{Smirnov},
and confirmed by the programme pursued by Lawler, Schramm, Werner,
Camia, Newman and others to understand SLE curves and conformal ensembles.

Only \emph{two-dimensional} lattices are considered in the remainder of Section~\ref{sec:bp}.

\subsection{Box-crossing property}\label{sec:bxp}

Loosely speaking, the  `\bxp' is the property that the probability of
an open crossing of a box with given aspect-ratio is bounded away from 0, uniformly in the
position, orientation, and size of the box.

Let $\sL=(V,E)$ be a planar lattice drawn in $\RR^2$, and let $\PP$
be a probability measure on $\Om =\{0,1\}^E$. For definiteness, we may think of $\sL$ as one of the square,
triangular, and hexagonal lattices, but the following discussion is valid in much greater generality.

Let $R$ be a (non-square) rectangle of $\RR^2$.
A lattice-path $\pi$ is said to \emph{cross} $R$
if $\pi$ contains an arc (termed a \emph{box-crossing}) lying in the
interior of $R$ except for its two endpoints, which are required
to lie, respectively, on the two shorter sides of $R$.
Note that a box-crossing of a rectangle lies in the longer direction.

Let $\om \in \Om$.
The rectangle $R$ is said to \emph{possess an open crossing}
if there exists an open box-crossing of
$R$, and we write $C(R)$  for this event.
Let $\sT$ be the set of translations of $\RR^2$, and $\tau\in\sT$.
Fix the aspect-ratio $\rho>1$.
Let $H_n=[0,\rho n]\times[0,n]$ and $V_n=[0,n]\times[0,\rho n]$, and
let $n_0=n_0(\sL)<\oo$ be minimal with the property that,
for all $\tau\in\sT$ and all $n \ge n_0$, $\tau H_n$ and $\tau V_n $
possess crossings in $\sL$.
Let
\begin{align}
  \b_\rho(\sL,\PP) &= \inf\Bigl\{ \PP(C(\tau H_n)),\PP(C( \tau V_n)): n \ge n_0,\ \tau\in\sT\Bigr\}.
\label{G13}
\end{align}
The pair $(\sL, \PP)$  is said to have  the $\rho$-\emph{\bxp}
if  $\b_\rho(\sL,\PP)>0$.

The measure $\PP$ is called \emph{positively associated} if,
for all increasing cylinder events $A$, $B$,
\begin{equation}\label{pos-ass}
\PP(A \cap B) \ge \PP(A) \PP(B).
\end{equation}
(See \cite[Sect. 4.2]{Grimmett_Graphs}.)
The value of $\rho$ in the \bxp\ is in fact
immaterial, so long as  $\rho > 1$ and $\PP$ is positively associated.
We state this explicitly as a proposition since we shall need it in Section~\ref{sec:rcm}.
The proof is left as an exercise (see \cite[Sect. 5.5]{Grimmett_Graphs}).

\begin{prop}\label{prop:rsw}
Let $\PP$ be a probability measure on $\Om$ that is positively associated.
If there exists $\rho>1$ such that $(\sL,\PP)$ has the $\rho$-\bxp,
then $(\sL,\PP)$ has the $\rho$-\bxp\  for all $\rho>1$.
\end{prop}

It is standard that the percolation measure (and more generally
the \rc\ measure of Section~\ref{sec:rcm},
see \cite[Sect. 3.2]{Grimmett_RCM})
are positively associated, and thus we may speak simply of the \bxp.

Here is a reminder about duality for planar graphs.
Let $G=(V,E)$ be a planar graph, drawn in the plane. Loosely speaking, the {\it planar
dual\/}  $G_\td$ of $G$ is the graph constructed by placing a
vertex inside every face of $G$ (including the infinite face if it exists)
and joining two such vertices by an edge $e_\td$ if and only if the
corresponding faces of $G$ share a boundary edge $e$.
The edge-set $E_\td$ of $G_\td$ is in one--one correspondence ($e \lra e_\td$) with $E$.
The duals of the square, triangular, and hexagonal lattices are
illustrated in Figure~\ref{fig:lattices}.

Let $\Om=\{0,1\}^E$, and $\om \in \Om$. With $\om$ we associate a configuration $\om_\td$ in
the dual space $\Om_\td = \{0,1\}^{E_\td}$ by $\om(e) + \om_\td(e_\td) = 1$.
Thus, an edge of the dual is open if and only if it crosses a closed edge of the primal graph.
The measure $\PP_p$ on $\Om$ induces the measure $\PP_{1-p}$ on $\Om_\td$.

The \bxp\ is fundamental to rigorous study of percolation in two dimensions. When it holds,
the process is either critical or supercritical.
If both $(\sL,\PP_p)$ and its dual $(\sL_\td, \PP_{1-p})$ have the \bxp,
then each is critical (see, for example, \cite[Props 4.1, 4.2]{GM1}).
The \bxp\ was developed by Russo \cite{Russo},
and Seymour and Welsh \cite{Seymour-Welsh},
and exploited heavily by Kesten \cite{Kesten_book}.
Further details of the use of the \bxp\ may be found in
\cite{CN06,Smirnov,Sun11,WW_park_city}.

One way of estimating the chance of a box-crossing is via its derivative. Let $A$ be an increasing
cylinder event,
and let $g(p) = \PP_p(A)$. An edge $e$ is called \emph{pivotal} for $A$ (in a configuration $\om$) if
$\om^e \in A$ and $\om_e \notin A$, where $\om^e$ (\resp, $\om_e$) is the configuration
$\om$ with the state of $e$ set to $1$ (\resp, $0$).
The so-called `Russo formula' provides a geometric representation for the derivative $g'(p)$:
$$
g'(p) = \sum_{e\in E}  \PP_p(e \text{ is pivotal for }A).
$$
With $A$ the event that the rectangle $R$ possesses an open-crossing, the edge $e$ is pivotal for $A$ if
the picture of Figure~\ref{fig:piv} holds.
Note the four `arms' centred at $e$, alternating primal/dual.

\begin{figure}[!b]
 \centering
    \includegraphics[width=0.5\textwidth]{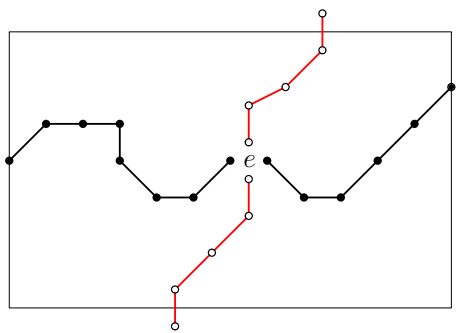}
  \caption{Primal and dual paths in the rectangle $R$.
A black path is an open primal path joining an endpoint
of $e$ to a left/right side of $R$, and a red path is an open dual path
from an endpoint of $e_\td$ to the top/bottom side of $R$.
An edge $e$ is pivotal for the box-crossing event if and only if
there are four arms of alternating type from $e$ to the boundary of the box.}
  \label{fig:piv}
\vspace*{3pt}
\end{figure}

It turns out that the nature of the percolation singularity is
partly determined by the asymptotic behaviour of the probability of such a `four-arm event'
at the critical point. This event has an associated critical exponent which we introduce next.

Let $\La_n$ be the set of vertices within graph-theoretic
distance $n$ of the origin  $0$, with boundary $\pd\La_n = \La_n \sm \La_{n-1}$.
Let $\Ann(N,n) = \La_n \setminus \La_{N-1}$ be the \emph{annulus} centred at $0$.
We  call $\pd \La_n$ (\resp, $\pd \La_N$) its \emph{exterior} (\resp, \emph{interior}) \emph{boundary}.

Let $k \in \NN$, and let
$\si=(\si_1,\si_2,\dots,\si_k) \in\{0,1\}^k$; we call $\si$ a \emph{colour sequence}. The sequence $\si$ is called \emph{monochromatic}
if either $\si=(0,0,\dots,0)$ or $\si=(1,1,\dots,1)$, and \emph{bichromatic}
otherwise. If $k$ is even, $\si$ is called \emph{alternating} if either $\si=(0,1,0,1,\dots)$
or $\si=(1,0,1,0,\dots)$. An open path of the primal (\resp, dual) lattice is said to have colour 1 (\resp,
0).
For $0<N<n$,
the arm event $A_{\si}(N,n)$
is the event that the inner boundary of
$\sA(N,n)$ is connected to the outer boundary by $k$ vertex-disjoint
paths
with colours $\sigma_1, \ldots, \sigma_k$, taken in anticlockwise order.

The choice of $N$ is largely immaterial to the asymptotics as $n\to\oo$, and
it is enough to take $N=N(\si)$ sufficiently
large that, for $n \ge N$, there exists a configuration with the required $j$ coloured paths.
It is believed that there exist \emph{arm exponents} $\rho(\si)$ satisfying
$$
\PP_{\pc}[A_{\sigma}(N,n)] \approx n^{-\rho(\si)} \qq \text{as } n \to\oo.
$$
Of particular interest here are the alternating arm exponents. Let $j \in \NN$, and
write $\rho_{2j} = \rho(\si)$ with $\si$ the alternating colour sequence of
length $2j$. Thus, $\rho_4$ is the exponent associated with the derivative
of box-crossing probabilities. Note that the radial exponent  $\rho$ satisfies
$\rho=1/\rho(\{1\})$.

\subsection{Star--triangle transformation}\label{sec:stt}

In its base form, the \stt\  is
a simple graph-theoretic relation.
Its principal use has been to explore models with characteristics
that are invariant under such transformations.
It  was discovered in the context of electrical networks
by Kennelly \cite{Ken} in 1899, and
it was adapted in 1944 by Onsager \cite{OnsI} to the Ising model in
conjunction with Kramers--Wannier duality.
It is a key element in the work
of Baxter \cite{Baxter_book} on exactly solvable models in statistical mechanics, and it has
become known as the \emph{Yang--Baxter equation} (see \cite{Perk-AY}
for a history of its importance in physics).
Sykes and Essam \cite{Sykes_Essam} used the \stt\
to predict the critical surfaces of inhomogeneous bond percolation on
triangular and hexagonal lattices, and it is a tool in the study of
the \rc\ model \cite{Grimmett_RCM}, and the dimer model \cite{Ken02}.

Its importance for probability stems from the fact that a variety of probabilistic models are
conserved under this transformation, including critical percolation, Potts, and random-cluster
models. More specifically,
the \stt\  provides couplings of critical probability measures under which certain geometrical
properties of configurations (such as connectivity in percolation)
are conserved.

\begin{figure}[t]
 \centering
    \includegraphics[width=0.6\textwidth]{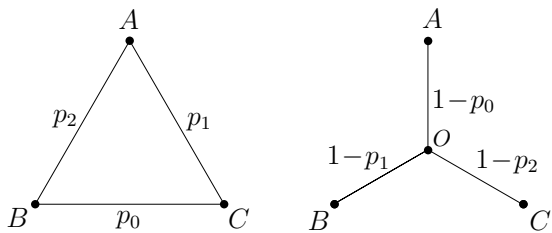}
  \caption{The star--triangle transformation}
  \label{fig:star_triangle_transformation}
\end{figure}

We summarize the \stt\ for percolation as in \cite[Sect.\ 11.9]{Grimmett_Percolation}.
Consider the triangle $G=(V,E)$
and the star $G'=(V',E')$ drawn in Figure~\ref{fig:star_triangle_transformation}. Let $\bp=(p_0,p_1,p_2)\in[0,1)^3$.
Write $\Om=\{0,1\}^E$ with associated (inhomogeneous)
product probability measure $\PP_\bp^\tri$
with intensities $(p_i)$ as illustrated,
and $\Om'=\{0,1\}^{E'}$ with associated measure $\PP_{1-\bp}^\hex$.
Let $\omega\in\Om$ and
$\om'\in\Om'$. The configuration $\om $ (\resp, $\om'$) induces a connectivity relation on the set $\{A,B,C\}$ within
$G$ (\resp, $G'$). It turns out that these two connectivity relations are equi-distributed so long
as $\kappa_\tri(\bp)=0$, where
\begin{equation}\label{kappadef}
\kappa_\tri(\bp) = p_0+p_1+p_2 - p_1p_2p_3 -1.
\end{equation}

This may be stated rigorously as follows.
Let $1(x \xleftrightarrow{G,\om} y)$ denote the indicator
function of the event that $x$ and $y$ are connected in $G$ by an open path of $\om$.
Thus, connections in $G$ are described by the family
$\{1(x \xleftrightarrow{G , \omega} y): x,y \in V\}$ of random variables, and similarly for $G'$.
It may be checked (or see \cite[Sect.\ 11.9]{Grimmett_Percolation}) that
the families
$$
\left\{1( x \xleftrightarrow{G, \omega} y) : x,y = A,B,C\right\}, \quad
\left\{1(x \xleftrightarrow{G', \om'} y) :x,y = A,B,C\right\},
$$
have the same law whenever $\kappa_\tri(\bp)=0$.

It is helpful to express this in terms of a coupling of $\PP_\bp^\tri$ and $\PP_{1-\bp}^\hex$.
Suppose $\bp\in[0,1)^3$ satisfies $\kappa_\tri(\bp)=0$, and let $\Om$ (\resp, $\Om'$)
have associated measure $\PP_\bp^\tri$ (\resp, $\PP_{1-\bp}^\hex$)
as above. There exist random mappings $T:\Om \to \Om'$ and
$S: \Om'\to\Om$ such that:
\begin{letlist}
\item
  $T(\omega)$ has the same law as $\om'$, namely $\PP_{1-\bp}^\hex$,
\item $S(\om')$ has the same law as $\omega$, namely $\PP_\bp^\tri$,
\item for  $x,y \in \{ A,B,C \}$, $1\bigl(x \xleftrightarrow{G,\omega} y\bigr) =
1\bigl(x \xleftrightarrow{G',T(\omega)} y\bigr)$,
\item for  $x,y \in \{ A,B,C \}$, $1\bigl(x \xleftrightarrow{G',\omega'} y\bigr) =
1\bigl(x \xleftrightarrow{G,S(\omega')} y\bigr)$.
\end{letlist}
Such mappings are
described informally in Figure~\ref{fig:coupling}
(taken from \cite{GM1}).

\begin{figure}[t]
  \begin{center}
    \includegraphics[width=1.0\textwidth]{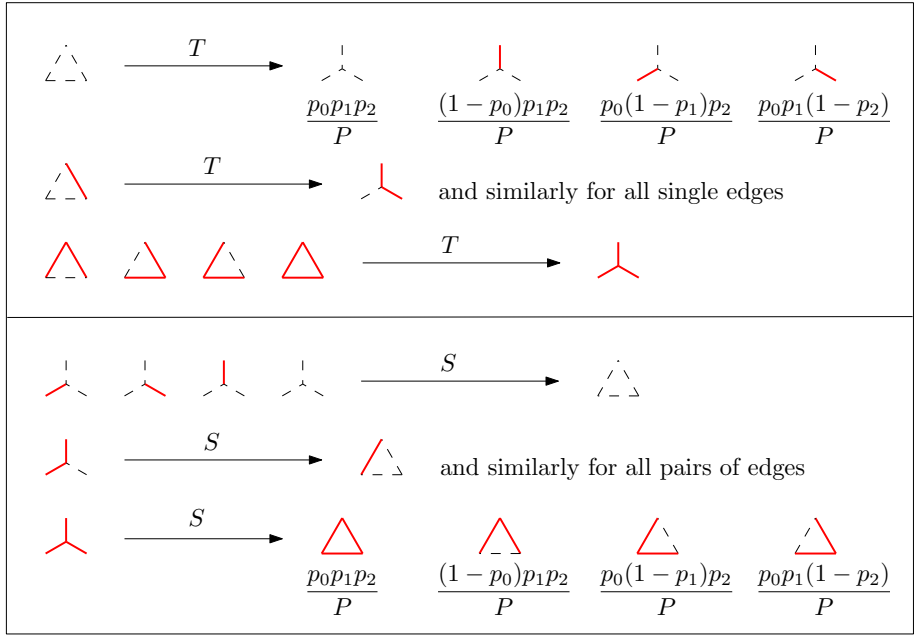}
  \end{center}
  \caption{The random maps $T$ and $S$ and their
transition probabilities, with $P:=(1-p_0)\times$ $(1-p_1)(1-p_2)$.
Note that $T(\om)$ is
deterministic for seven of the eight elements of $\Om$; only in the eighth case does $T(\om)$ involve
further randomness. Similarly, $S(\om')$ is deterministic except for one special $\om'$.}
  \label{fig:coupling}
\end{figure}

The \stt\ may evidently be used to couple bond percolation on the
triangular and hexagonal lattices. This may be done, for example, by
applying it to every upwards pointing triangle
of $\TT$. Its impact however extends much beyond this.
Whenever two percolation models are related by sequences of disjoint \stt s, their open connections
are also related (see \cite{GM3}).
That is, the \stt\ transports not only measures but also
open connections. We shall see how this may be used in the next section.

\subsection{Universality for bond percolation}\label{sec:univ}

The hypothesis of universality states in the context of percolation that the
nature of the singularity depends on the number of dimensions but not
further on the details of the model (such as choice of lattice, and whether bond or site).
In this section, we summarize results of \cite{GM1,GM2} showing a degree of universality
for a class of bond percolation models in two dimensions. The basic idea is as follows.
The \stt\ is a relation between a large family of critical bond percolation models.
Since it preserves open connections, these models have singularities of the same type.

We concentrate here on the square, triangular,
and hexagonal (or honeycomb) lattices, denoted respectively
as $\LL^2$, $\TT$, and $\HH$. The following analysis applies to a large class
of so-called isoradial graphs of which these lattices are examples (see \cite{GM3}).
The critical probabilities of homogeneous percolation on these lattices
are known as follows (see \cite{Grimmett_Percolation}):
$\pc(\LL^2) = \frac12$, and $\pc(\TT)=1-\pc(\HH)$ is the
root in the interval
$(0,1)$ of the cubic equation $3p-p^3-1=0$.

We define next \emph{inhomogeneous} percolation on these lattices.
The edges of the square lattice are partitioned into two classes (horizontal and vertical) of parallel edges,
while those of the triangular and hexagonal lattices may be split into three such classes.
The product measure on the edge-configurations is permitted to
have different intensities on different edges, while requiring that any two parallel edges have the same intensity.
Thus, inhomogeneous percolation on the square lattice has two parameters,
$p_0$ for horizontal edges and $p_1$ for vertical edges, and
we denote the corresponding measure $\PP_{\bp}^\squ$ where $\bp=(p_0,p_1)\in[0,1)^2$.
On the triangular and hexagonal lattices, the measure is defined by a triplet of parameters $\bp=(p_0, p_1, p_2)\in [0,1)^3$,
and we denote these measures $\PP_{\bp}^\tri$ and $\PP_{\bp}^\hex$, respectively.

Criticality is identified in an inhomogeneous model
via a `critical surface'. Consider bond percolation on
a lattice $\sL$ with edge-probabilities $\bp=(p_0,p_1,\dots)$.
The \emph{critical surface} is an equation of the form $\kappa(\bp)=0$,
where the percolation probability $\th(\bp)$ satisfies
$$
\theta(\bp) \begin{cases} = 0 &\text{if } \kappa(\bp)<0,\\
>0 &\text{if } \kappa(\bp)>0.
\end{cases}
$$
The discussion of Section~\ref{sec:pt} may be adapted to the critical
surface of an inhomogeneous model.

The critical surfaces of the above models are
given explicitly in \cite{Grimmett_Percolation, Kesten_book}.
Let
\begin{alignat*}{2}
\kappa_\squ(\bp) &= p_0+p_1-1, &&\quad \bp=(p_0,p_1), \\
\kappa_\tri(\bp) &= p_0+p_1+p_2 - p_0 p_1 p_2 - 1,&&\quad \bp=(p_0,p_1,p_2),\\
\kappa_\hex(\bp) &= -\kappa_\tri(1-p_0,1-p_1,1-p_2),&&\quad \bp=(p_0,p_1,p_2).
\end{alignat*}
The critical surface of the lattice $\LL^2$
(\resp, $\TT$, $\HH$)  is
given by $\kappa_\squ(\bp) = 0$ (\resp, $\kappa_\tri(\bp)=0$, $\kappa_\hex(\bp)=0$).

Let $\sM$ denote the set of all inhomogeneous bond percolation models
on the square, triangular, and hexagonal lattices,
with edge-parameters belonging to the half-open interval $[0,1)$
and lying in the appropriate critical surface.
A critical exponent $\pi$ is said to \emph{exist} for a model $M\in \sM$ if the
appropriate asymptotic relation of Table~\ref{Tab-ce} holds, and $\pi$ is called
\emph{$\sM$-invariant} if it exists for all $M \in \sM$ and its value
is independent of the choice of such $M$.

\begin{thm}[\cite{GM2}] \label{thm:eq}
\mbox{}
\begin{letlist}
\item For every $\pi \in \{\rho\}\cup \{\rho_{2j}: j \ge 1\}$, if $\pi$ exists for some model $M \in \sM$, then
it is $\sM$-invariant.
\item If either $\rho$ or $\eta$ exist for some $M\in\sM$, then $\rho$, $\delta$, $\eta$ are $\sM$-invariant
and  satisfy the scaling relations $2\rho=\de+1$, $\eta\rho=2$.
\end{letlist}
\end{thm}

Kesten \cite{Kesten87} showed\footnote{See also \cite{Nolin}.}
that the `near-critical' exponents $\beta$, $\gamma$, $\nu$, $\De$ may be
given explicitly in terms of $\rho$ and $\rho_4$, for two-dimensional models
satisfying certain symmetries. \emph{Homogeneous} percolation on our three lattices
have these symmetries, but it is not known whether the strictly \emph{inhomogeneous}
models have sufficient regularity for the conclusions to apply.
The next theorem is a corollary of Theorem~\ref{thm:eq} in the light of the results
of \cite{Kesten87,Nolin}.

\begin{thm}[\cite{GM2}]\label{thm:eq2}
Assume that $\rho$ and $\rho_4$ exist for some $M \in \sM$.
Then $\beta$, $\gamma$, $\nu$, and $\De$ exist
for homogeneous percolation on the square, triangular and hexagonal lattices, and they are invariant
across these three models. Furthermore, they satisfy the scaling relations
$\g+2\beta =\beta(\de+1)=2\nu$, $\De=\beta\de$.
\end{thm}

A key intermediate step in the proof of Theorem~\ref{thm:eq} is the \bxp\ for
inhomogeneous percolation on these lattices.

\begin{thm}[\cite{GM1}] \label{main_result}\mbox{}
\begin{letlist}
\item If $\bp \in (0,1)^2$ satisfies $\kappa_\squ(\bp)=0$,
  then $\PP_{\bp}^\squ$ has the \bxp.
\item If $\bp \in [0,1)^3$ satisfies $\kappa_\tri(\bp)=0$,
  then both $\PP_{\bp}^\tri$ and $\PP_{1-\bp}^\hex$
have the \bxp.
\end{letlist}
\end{thm}

In the remainder of this section, we outline the proof of Theorem~\ref{main_result} and indicate the
further steps necessary for Theorem~\ref{thm:eq}. The starting point is the observation
of Baxter and Enting \cite{Baxter_399} that the \stt\ may be
used to transform the square into the triangular lattice. Consider the `mixed lattice'  on the left of Figure
\ref{fig:lattice_trans} (taken from \cite{GM1}),
in which there is an interface $I$ separating the square
from the triangular parts. Triangular edges have length $\sqrt 3$
and vertical edges length $1$. We apply
the \stt\ to every upwards pointing triangle, and then to every downwards pointing star.
The result is a translate of the mixed lattice with the interface lowered by one step. When performed
in reverse, the interface is raised one step.

\begin{figure}[!b]
\vspace*{-3pt}
  \begin{center}
    \cpsfrag{tu}{$\Tu$}
    \cpsfrag{td}{$\Td$}
    \cpsfrag{su}{$\Su$}
    \cpsfrag{sd}{$\Sd$}
    \includegraphics[width=1.0\textwidth]{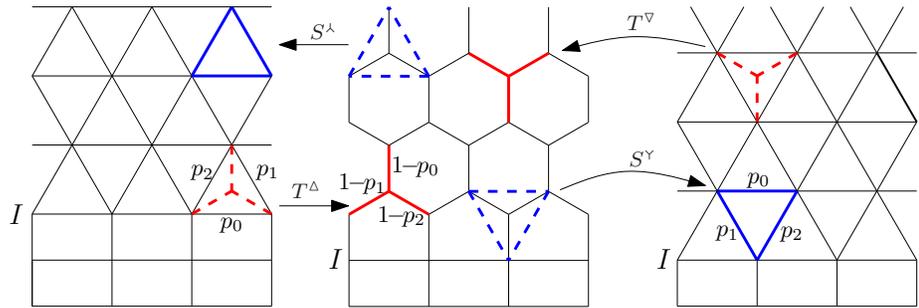}
  \end{center}
\vspace*{-6pt}
  \caption{Transformations $\Su$, $\Sd$, $\Tu$, and $\Td$ of mixed lattices.
    The transformations map the dashed stars/triangles to
    the bold stars/triangles.
    The interface-height decreases by $1$ from the leftmost to the rightmost graph.}
  \label{fig:lattice_trans}
\vspace*{-6pt}
\end{figure}

This star--triangle map is augmented with probabilities as follows.
Let $\bp=(p_0,p_1,p_2)\in
[0,1)^3$ satisfy $\kappa_\tri(\bp)= 0$.
An edge $e$ of the mixed lattice is declared \emph{open} with probability:
\begin{letlist}
\item $p_0$ if $e$ is horizontal,
\item $1- p_0$ if $e$ is vertical,
\item $p_1$ if $e$ is the right edge of an upwards pointing triangle,
\item $p_2$ if $e$ is the left edge of an upwards pointing triangle,
\end{letlist}
and the ensuing product measure is written $\PP_\bp$.
Write $\Sd \circ \Tu$ for the left-to-right map of Figure~\ref{fig:lattice_trans},
and $\Su \circ \Td$ for the right-to-left map.
As described in Section~\ref{sec:stt},
each $\tau \in\{\Sd \circ \Tu, \Su \circ \Td\}$  may be extended to maps between configuration spaces,
and they give rise to
couplings of the relevant probability measures under which local open connections are preserved.
It follows that, for a open path $\pi$  in the domain of $\tau$,
the image $\tau(\pi)$ contains an open path
$\pi'$
with endpoints within distance 1 of those of $\pi$, and furthermore every point of $\pi'$ is within
distance $1$
of some point in $\pi$. We shall speak of $\pi$ being transformed to $\pi'$ under $\tau$.

Let $\alpha >2$ and let $R_N$ be a $2\alpha N \times N$ rectangle in the square part of a mixed lattice.
Since $\PP_\bp$ is a product measure, we may take as interface
the set $\ZZ\times\{N\}$.
Suppose there is an open path $\pi$ crossing $R_N$ horizontally.
By making $N$ applications of $\Sd \circ \Tu$,
$\pi$ is transformed into an open path $\pi'$ in the triangular part of the lattice. As above, $\pi'$ is
within distance $N$ of $\pi$, and its endpoints are within distance $N$ of those of $\pi$. As illustrated
in Figure~\ref{fig:hsq2tr}, $\pi'$ contains a horizontal crossing of a $2(\alpha-1) \times 2N$ rectangle $R'_N$
in the triangular lattice. It follows that
$$
\PP^\tri_\bp(C(R'_N)) \ge \PP_{(p_0,1-p_0)}^\squ(C(R_N)), \qq N \ge 1.
$$
This is one of two inequalities that jointly imply that,
if $\PP_{(p_0,1-p_0)}^\squ$ has the \bxp\, then so does $\PP_\bp^\tri$.
The other such inequality is concerned with vertical crossings of rectangles.
It is not so straightforward to derive, and makes use of a probabilistic estimate
based on the randomization within the map  $\Sd \circ \Tu$ given in Figure~\ref{fig:coupling}.

\begin{figure}[t]
\centering
    \includegraphics[width=0.9\textwidth]{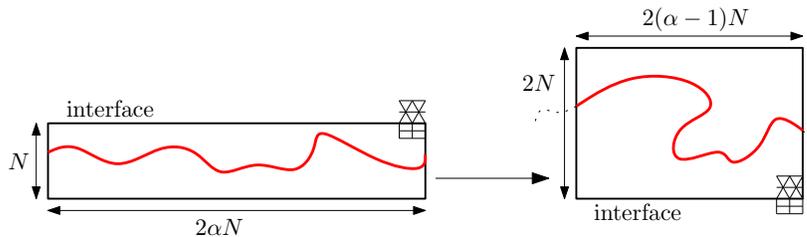}
  \caption{After $N$ applications of $\Sd \circ \Tu$, a horizontal crossing of a rectangle of $\ZZ^2$
has been transformed into a crossing of a
rectangle in $\TT$ with altered aspect-ratio.}
  \label{fig:hsq2tr}
\end{figure}

One may similarly show that $\PP_{(p_0,1-p_0)}^\squ$ has the \bxp\ whenever $\PP_\bp^\tri$ has it.
As above, two inequalities are needed, one of which is simple and the other less so. In summary, $\PP_{(p_0,1-p_0)}^\squ$
has the \bxp\ if and only if $\PP_\bp^\tri$ has it. The
reader is referred to \cite{GM1} for the details.

Theorem~\ref{main_result} follows thus. It was shown by
Russo \cite{Russo} and by Seymour and Welsh \cite{Seymour-Welsh}
that $\PP_{(\frac12,\frac12)}^\squ$ has the \bxp\ (see also \cite[Sect. 5.5]{Grimmett_Graphs}).
By the above, so does $\PP^\tri_\bp$ for $\bp=(\frac12,p_1,p_2)$
whenever $\kappa_\tri(\frac12,p_1,p_2)=0$.
Similarly, so does $\PP_{(p_1,1-p_1)}^\squ$, and therefore also $\PP_\bp^\tri$
for any triple $\bp=(p_0,p_1,p_2)$ satisfying $\kappa_\tri(\bp)=0$.

\begin{figure}[t]
  \begin{center}
    \cpsfrag{tu}{$T^-$}
    \cpsfrag{td}{$T^+$}
    \cpsfrag{su}{$S^+$}
    \cpsfrag{sd}{$S^-$}
    \includegraphics[width=1.0\textwidth]{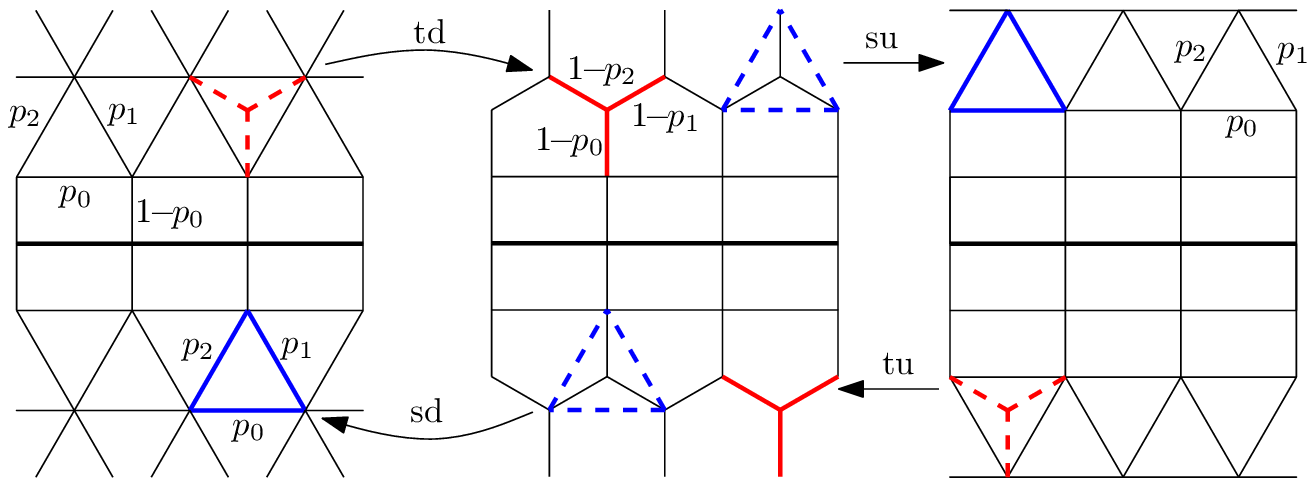}
  \end{center}
\vspace*{-6pt}
  \caption{The transformation
$S^+ \comp T^+$ (\resp, $S^- \comp T^-$)
transforms $\Lattice^1$ into $\Lattice^2$
(\resp, $\Lattice^2$ into $\Lattice^1$).
    They map the dashed graphs to
    the bold graphs. }
  \label{fig:lattice_strip}
\end{figure}

We close this section with some notes on the further steps required for Theorem~\ref{thm:eq}.
We restrict ourselves to a consideration of the radial exponent $\rho$, and the reader
is referred to \cite{GM2} for the alternating-arm exponents.
Rather than the mixed lattices of Figure~\ref{fig:lattice_trans}, we consider
the hybrid lattices $\LL^m$ of Figure~\ref{fig:lattice_strip}
having a band of square lattice of width $2m$, with triangular sections above and below.
The edges of triangles have length $\sqrt 3$ and the vertical edges  length $1$.
The edge-probabilities of $\LL^m$ are as above, and the resulting measure is
denoted~$\PP_\bp^m$.

\begin{figure}[t]
\centering
    \includegraphics[width=0.5\textwidth]{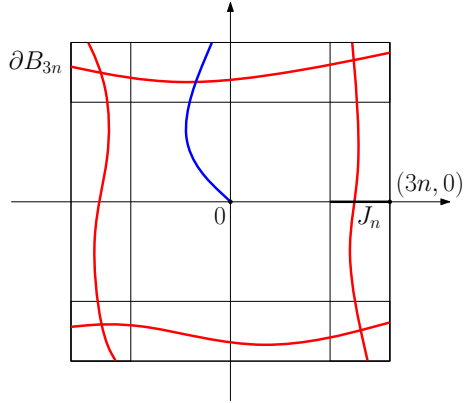}
  \caption{If $0$ is connected to $\pd B_{3n}$ and the four
box-crossings occur, then $0$ is connected to the line-segment $J_n$. }
  \label{fig:radius}
\end{figure}


Let $n \ge 3$,  and $B_n = [-n,n]^2 \subseteq \RR^2$, and write
$A_n=\{0\lra \pd B_n\}$ where $\pd B$ denotes the boundary
of the box $B$.
Let $J_n$ be the line-segment $[2n,3n]\times \{0\}$, and note that $J_n$
is invariant under the lattice
transformations of Figure~\ref{fig:lattice_strip}.
If $A_{3n}$ occurs, and in addition the four rectangles illustrated in Figure~\ref{fig:radius}
have crossings, then $0 \lra J_n$.
Let $\bp\in[0,1)^3$. By Theorem~\ref{main_result} and
positive association, there exists $\a>0$ such that,
for $n \ge 3$,
\begin{align}
\a \PP^\squ_\bp(A_{3n}) &\le  \PP^\squ_\bp(0 \lra J_n) \le  \PP_\bp^\squ(A_{2n}),\label{g10}\\
\a \PP^\tri_\bp(A_{3n}) &\le   \PP^\tri_\bp(0 \lra J_n) \le  \PP_\bp^\tri(A_{2n}).\label{g11}
\end{align}

By making $3n$ applications of the mapping $S^+\circ T^+$ (\resp, $S^-\circ T^-$)
of Figure~\ref{fig:lattice_strip},
we find that
$$
\PP_{(p_0,1-p_0)}^\squ(0 \lra J_n) = \PP_\bp^{3n}(0\lra J_n) = \PP_\bp^0(0\lra J_n)
=\PP_\bp^\tri(0 \lra J_n).
$$
The equality of the exponents  $\rho$ (if they exist) for
these two models follows by \eqref{g10}--\eqref{g11},
and the proof of Theorem~\ref{thm:eq}(a)
is completed in similar manner to that of the \bxp, Theorem~\ref{main_result}.

Part (b) is a consequence of Theorem~\ref{main_result}
on applying  Kesten's results of  \cite{Kes87a} about scaling relations at criticality.

\section{Random-cluster model}\label{sec:rcm}

\subsection{Background}

Let $G=(V,E)$ be a finite graph, and  $\Om=\{ 0,1\}^E$. For
$\om\in\Om$, we write $\eta(\om)=\{e\in E: \om(e)=1\}$ for the set of
open edges, and $k(\om)$ for the number of connected components, or
`open clusters', of the subgraph $(V,\eta(\om))$.
The {\it random-cluster measure\/} on $\Om$,
with parameters $p\in[0,1]$, $q\in(0,\oo)$ is the probability measure
\begin{equation}\label{rcmeas}
\fpq (\om )=\frac{1}{Z}\,\biggl\{\prod_{e\in E} p^{\om (e)}
(1-p)^{1-\om (e)}\biggr\} q^{k(\om )} ,\qq\om\in\Om,
\end{equation}
where $Z=Z_{G,p,q}$ is the normalizing constant.
We assume throughout this section that $q \ge 1$, and
for definiteness shall work only with
the hypercubic lattice $\LL^d=(\ZZ^d, \EE^d)$ in $d \ge 2$ dimensions.

This measure was introduced by Fortuin and Kasteleyn in
a series of papers around 1970, in a unification
of electrical networks, percolation, Ising, and Potts models.
Percolation is retrieved by setting $q=1$,
and electrical networks
arise via the limit $p,q\to 0$ in such a way that $q/p\to 0$.
The relationship to Ising/Potts models is more complex
in that it involves a transformation
of measures. In brief, two-point connection
probabilities for the \rc\ measure with $q\in\{2,3,\dots\}$
correspond to correlations for ferromagnetic Ising/Potts models,
and this allows a geometrical interpretation
of their correlation structure.
A fuller account of the \rc\ model and its history and associations may be found in \cite{Grimmett_RCM}.

We omit an account of the properties of \rc\ measures, instead referring the reader to
\cite{Grimmett_RCM,Grimmett_Graphs}. Note however that \rc\ measures
are positively associated whenever $q\ge 1$, in that \eqref{pos-ass} holds for all
pairs $A$, $B$ of increasing events.

The \rc\ measure may not be defined directly on the hypercubic
lattice $\LL^d=(\ZZ^d,\EE^d)$, since this is infinite.
There are two possible ways to proceed, of which we choose here
to use weak limits. Towards this end we
introduce boundary conditions.
Let $\La$ be a finite box in $\ZZ^d$. For $b\in\{ 0,1\}$, define
$$
\Om_\La^b=\{\om\in\Om :\om (e)=b \text{ for } e\notin\EE_\La\} ,
$$
where $\EE_\La$ is the set of edges of $\LL^d$ joining pairs of vertices
belonging to $\La$. Each of the two values of $b$ corresponds to a certain `boundary
condition' on $\La$, and we shall be interested in
the effect of these boundary conditions in the infinite-volume limit.

On $\Om_\La^b$, we define a random-cluster measure
$\phi_{\La ,p,q}^b$ as follows. Let
\begin{equation}\label{13.6}
\phi_{\La ,p,q}^b(\om )=\frac{1}{Z_{\La
,p,q}^b}\,\Biggl\{\prod_{e\in\EE_\La} p^{\om (e)} (1-p)^{1-\om
(e)}\Biggr\} q^{k(\om ,\La )},\qq \om\in\Om_\La^b,
\end{equation}
where $k(\om ,\La )$ is the number of clusters of $(\ZZ^d,\eta (\om ))$
that intersect $\La$. The boundary condition $b=0$ (\resp,
$b=1$) is sometimes termed `free' (\resp, `wired'). The choice of boundary condition
affects the measure through the total number $k(\om,\La)$ of open clusters: when using the wired
boundary condition, the set of  clusters intersecting the boundary of $\La$
contributes only $1$ to this total.

The free/wired boundary conditions are extremal within a broader class.  A boundary
condition on $\La$ amounts to a rule for how to count the clusters intersecting
the boundary $\pd\La$ of $\La$. Let $\xi$ be an equivalence relation  on $\pd\La$; two vertices
$v,w\in\pd \La$ are identified as a single point if and only if $v \xi w$.
Thus $\xi$ gives rise to a cluster-counting function $K(\cdot,\xi)$, and thence a probability
measure $\phi_{\La,p,q}^\xi$ as in \eqref{13.6}. It is an exercise in Holley's
inequality \cite{Hol} to show that
\begin{equation}\label{G11-}
\phi_{\La,p,q}^\xi \lest \phi_{\La,p,q}^{\xi'} \qq \text{if } \xi \le \xi',
\end{equation}
where we write $\xi\le \xi'$ if, for all pairs $v$, $w$, $v\xi w\ \Rightarrow\ v \xi' w$.
In particular,
\begin{equation}\label{G11}
\phi_{\La,p,q}^0 \lest \phi_{\La,p,q}^\xi \lest \phi_{\La,p,q}^1\qq\text{for all } \xi.
\end{equation}

We may now take the infinite-volume limit.
It may be shown that the weak limits
$$
\fpqb =\lim_{\La\to\ZZ^d} \phi_{\La ,p,q}^b,\qq b=0,1,
$$
exist, and are translation-invariant and ergodic (see \cite{G93}).
The limit measures, $\fpqo$ and $\fpqon$, are called `random-cluster measures'
on $\LL^d$, and they are extremal in the following sense.
There is a larger family of measures that can be constructed on $\Om$, either
by a process of weak limits, or by the procedure that gives rise to so-called DLR measures
(see \cite[Chap. 4]{Grimmett_RCM}).
It turns out that $\fpqo \lest \fpq \lest \fpqon$
for any such measure $\fpq$, as in \eqref{G11}.
Therefore, there exists a unique \rc\ measure if and only if
$\fpqo=\fpqon$.

The {\it percolation probabilities} are defined by
\begin{equation}\label{13.8}
\th^b(p,q)=\fpqb (0\lra\infty ),\qq b=0,1,
\end{equation}
and the {\it critical values\/} by
\begin{equation}\label{rccritprob}
\pcb (q)=\sup\{ p:\th^b(p,q)=0\} ,\qq b=0,1.
\end{equation}
We are now ready to present a theorem that gives sufficient conditions under which
$\fpqo=\fpqon$. The proof may be found  in \cite{Grimmett_RCM}.

\begin{thm}\label{rcconvexity}
Let $d\geq 2$ and $q\geq 1$. We have that\/{\rm:}
\begin{letlist}
\item {\rm\cite{ACCN}}  $\fpqo =\fpqon$ if $\th^1(p,q)=0$,
\item {\rm\cite{G93}} there exists a countable subset $\sD_{d,q}$
of\/ $[0,1]$, possibly empty, such that
$\fpqo =\fpqon$ if and only if  $p\notin\sD_{d,q}$.
\end{letlist}
\end{thm}

By Theorem~\ref{rcconvexity}(b), $\th^0(p,q)=\th^1(p,q)$ for $p\notin\sD_{d,q}$,
whence $\pco (q)=\pcon (q)$ by the monotonicity of the $\th^b(\cdot,q)$. Henceforth we refer
to the critical value as $\pc (q)$.
It is a basic fact that $\pc(q)$ is non-trivial, which is to say that
$0<\pc(q)<1$ whenever $d \ge 2$ and $q \ge 1$.

It is an open problem to find a satisfactory definition of $\pc(q)$ for
$q<1$. Despite the failure of positive association in this case,
it may be shown by the so-called  `comparison inequalities'
(see \cite[Sect.\ 3.4]{Grimmett_RCM})
that there exists no infinite cluster for $q\in(0,1)$ and small $p$, whereas there is
an infinite cluster for $q\in(0,1)$ and large $p$.

The following is an important conjecture ccncerning the discontinuity set $\sD_q$.

\begin{conj}\label{uniqconj}
There exists $Q(d)$ such that\/{\rm:}
\begin{letlist}
\item if $q<Q(d)$, then $\th^1(\pc,q)=0$ and $\sD_{d,q}=\es$,
\item if $q>Q(d)$, then $\th^1(\pc,q)>0$ and $\sD_{d,q}=\{\pc\}$.
\end{letlist}
\end{conj}

In the physical vernacular, there is conjectured to exist a critical value
of $q$ beneath which the phase transition is continuous (`second order')
and above which it is discontinuous (`first order').
It was proved in \cite{Kot-S, LMMRS}
that there is a first-order transition for large $q$, and it is expected  that
$$
Q(d)= \begin{cases} 4 &\text{if } d=2,\\
2 &\text{if } d\ge 6.
\end{cases}
$$
This may be contrasted with the best current rigorous estimate
in two dimensions, namely $Q(2) \le 25.72$, see \cite[Sect. 6.4]{Grimmett_RCM}.
Recent progress towards a proof that $Q(2)=4$ is found in \cite{BDS2}.

The third result of this article concerns the behaviour of the \rc\ model on the square lattice $\LL^2$,
and particularly its critical value.

\subsection{Critical point on the square lattice}\label{sec:critpt}

\begin{thm}[\cite{Beffara_Duminil}]\label{thm@pc}
The \rc\ model on $\LL^2$ with cluster-weighting factor $q \ge 1$ has
critical value
$$
\pc(q) = \frac {\sqrt q}{1+\sqrt q}.
$$
\end{thm}

This exact value has been `known' for a long time, but the full proof has
been completed only recently.
When $q=1$, the statement $\pc(1)=\frac12$ is the Harris--Kesten theorem for bond percolation.
When $q=2$, it amounts to the well known
calculation of the critical temperature of the
Ising model. For large $q$, the result was proved  in \cite{LMMRS, LMR}.
There is a `proof' in the physics paper \cite{HKW} when $q \ge 4$.
It has been known since \cite{G93,Welsh93} that $\psd(q) \ge \sqrt q/(1+\sqrt q)$.

\begin{figure}[t]
  \begin{center}
    \includegraphics[width=0.6\textwidth]{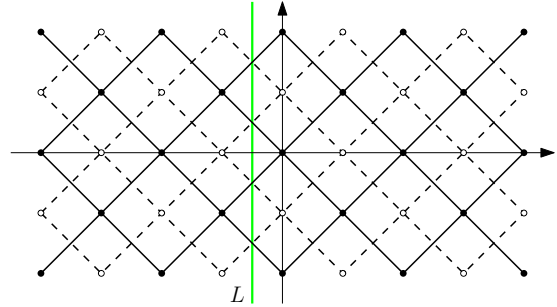}
  \end{center}
  \caption{The lattice $\LL^2$ and its dual, rotated through $\pi/4$. Under reflection in the green line $L$,
the primal is mapped to the dual.}
  \label{fig:z2rot}
\end{figure}

The main complication of Theorem~\ref{thm@pc}
beyond the $q=1$ case stems from the interference of boundary conditions
in the proof and applications of the \bxp, and this is where the major contributions of \cite{Beffara_Duminil}
are to be found. We summarize first the statement of the \bxp; the proof
is outlined in Section~\ref{sec:pf-bxp}. It is convenient
to work on the square lattice $\LL^2$ rotated through $\pi/4$, as illustrated in Figure~\ref{fig:z2rot}.
For $a=(a_1,a_2)\in \NN^2$ and $b=(b_1,b_2)\in\NN^2$,  the
\emph{rectangle} $R_{a,b}$ of this graph is the subgraph induced by the vertices lying inside
the rectangle $[a_1,a_2] \times [b_1,b_2]$ of $\RR^2$.
We shall consider two types of boundary condition on
$R_{a,b}$. These affect the counts  of clusters, and therefore the measures.
\begin{list}{}{\setlength{\labelwidth}{2.3cm}\setlength{\leftmargin}{2.8cm}}
\item[\emph{Wired} ($1$):] all vertices in the boundary of a rectangle are identified as a single vertex.
\item[\emph{Periodic} ($\tp$):] each vertex $(a_1,y)$ (\resp, $(x, b_1)$) of the boundary  of $R_{a,b}$ is wired to
$(a_2, y)$ (\resp, $(x,b_2)$).
\end{list}

Let $q \ge 1$, and write $\psd=\sqrt q/(1+\sqrt q)$ and $B_m=R_{(-m,-m),(m,m)}$.
The suffix in $\psd$ stands for `self-dual', and its use is explained
in the next section. The \rc\ measure on
$B_m$ with parameters $p$, $q$ and boundary condition $b$ is denoted $\phi_{p,m}^b$.
For a rectangle $R$, we write $\Ch(R)$ (\resp, $\Cv(R)$) for the event that $R$ is crossed
horizontally (\resp,
vertically) by an open path.

\begin{prop}[\cite{Beffara_Duminil}]\label{thm:rsw-rc}
There exists $c=c(q)>0$ such that, for $m > \frac32 n > 0$,
$$
\phi_{\psd,m}^\tp\left[\Ch\bigl([0,\tfrac32 n)\times[0,n)\bigr)\right] \ge c.
$$
\end{prop}

The choice of periodic boundary condition is significant
in that the ensuing measure is translation-invariant.
Since the measure is invariant also under rotations through $\pi/2$, this inequality
holds also for crossings of vertical boxes.
Since \rc\ measures are positively associated,  by Proposition~\ref{prop:rsw},
the measure $\phi_{\psd,m}^\tp$ satisfies a `finite-volume' $\rho$-\bxp\ for all $\rho>1$.

An infinite-volume version of Proposition~\ref{thm:rsw-rc} will be useful later.
Let $m >\frac32 n \ge 1$. By stochastic ordering \eqref{G11},
\begin{align}
\phi_{\psd,m}^1\left[\Ch\bigl([0,\tfrac32 n)\times[0,n)\bigr)\right]
&\ge \phi_{\psd,m}^\tp\left[\Ch\bigl([0,\tfrac32 n)\times[0,n)\bigr)\right]\nonumber\\
&\ge c. \label{bf9-}
\end{align}
Let $m \to\oo$ to obtain
\begin{equation}\label{bf9}
\phi_{\psd,q}^1\left[\Ch\bigl([0,\tfrac32 n)\times[0,n)\bigr)\right]
\ge c, \qq n \ge 1.
\end{equation}
By Proposition~\ref{prop:rsw}, $\phi^1_{\psd,q}$ has
the \bxp.
Equation \eqref{bf9} with $\phi_{\psd,q}^0$ in place of $\phi_{\psd,q}^1$ is false
for large $q$, see \cite[Thm 6.35]{Grimmett_RCM}.

Proposition~\ref{thm:rsw-rc} may be used to show also
the exponential-decay of connection probabilities when $p<\psd$.
See \cite{Beffara_Duminil}
for the details.

This section closes with a note about other two-dimensional models.
The proof of Theorem~\ref{thm@pc}
may be adapted (see \cite{Beffara_Duminil})
to the triangular and hexagonal lattices, thus
complementing known inequalities of \cite[Thm 6.72]{Grimmett_RCM}
for the critical points.  It is an open problem to prove the
conjectured critical surfaces of inhomogeneous models on
$\LL^2$, $\TT$, and $\HH$. See \cite[Sect. 6.6]{Grimmett_RCM}.

\subsection{Proof of the \bxp}\label{sec:pf-bxp}

We outline the proof of Proposition~\ref{thm:rsw-rc},
for which full details may be found  in \cite{Beffara_Duminil}.
There are two steps: first, one uses duality to prove  inequalities
about crossings of certain regions;
secondly, these are used to estimate the probabilities of crossings of rectangles.

\smallskip\noindent
\emph{Step 1, duality}.
Let $G=(V,E)$ be a finite, connected planar graph embedded in $\RR^2$,
and let $G_\td=(V_\td,E_\td)$ be its
planar dual graph.
A configuration $\om\in\{0,1\}^E$ induces a configuration $\om_\td \in\{0,1\}^{E_\td}$
as in Section~\ref{sec:bxp}
by $\om(e) + \om_\td(e_\td) = 1$.

We recall the use of duality for bond percolation on $\LL^2$: there
is a horizontal  open primal crossing of the rectangle
$[0,n+1]\times[0,n]$ (with the usual lattice orientation)
if and only if there is no vertical open dual crossing of the dual rectangle.
When $p=\frac12$, both rectangle and probability
measure are self-dual, and thus the  chance
of a primal crossing is $\frac12$, whatever the value of
$n$.  See \cite[Lemma 11.21]{Grimmett_Percolation}.

Returning to the \rc\ measure on $G$, if $\om$ has law $\phi_{G,p,q}$, it may be shown using
Euler's formula (see \cite[Sect. 6.1]{Grimmett_RCM} or \cite[Sect. 8.5]{Grimmett_Graphs})
that $\om_\td$ has law $\phi^1_{G_\td,p_\td,q}$ where
$$
\frac p{1-p} \cdot\frac {p_\td}{1-p_\td} = q.
$$
Note that $p=p_\td$ if and only if $p=\psd$.
One must be careful with boundary conditions.
If the primal measure has the free boundary condition, then the dual
measure has the wired boundary condition (in effect, since $G_\td$ possesses a
vertex in the infinite face of $G$).

Overlooking temporarily the issue of boundary condition,
the dual graph of a rectangle $[0,n)^2$ in the square lattice is a
rectangle in a shifted square lattice, and this leads to
the aspiration to find a self-dual measure and a crossing event with probability
bounded away from $0$ uniformly in $n$. The natural measure is
$\phi_{\psd,m}^\tp$, and the natural event is $\Ch\bigl([0,n)^2\bigr)$.
Since this measure is defined on a torus, and tori are not planar, Euler's formula cannot be applied
directly. By a consideration of the homotopy of the torus, one
obtains via an amended Euler formula that there exists $c_1=c_1(q)>0$ such that
\begin{equation}\label{bf2}
\phi_{\psd,m}^\tp\left[\Ch\bigl([0,n)^2\bigr)\right] \ge c_1, \qq 0 <n < m.
\end{equation}

\begin{figure}[t]
\centering
    \includegraphics[width=0.8\textwidth]{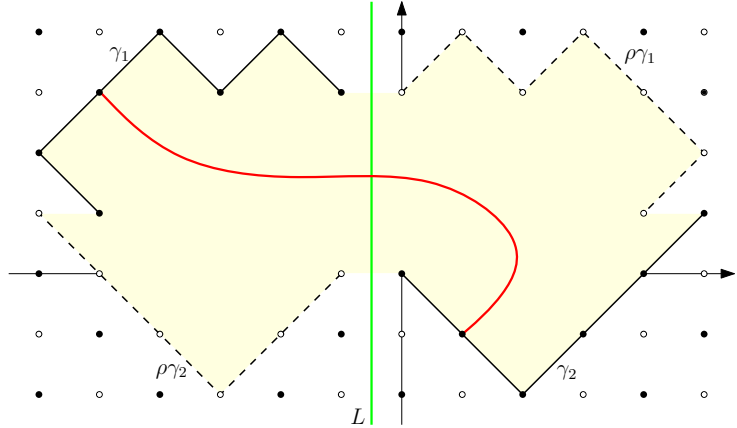}
  \caption{Under reflection $\rho$ in the green line $L$, the primal lattice is mapped to the dual.
The primal path $\gamma_1$ is on the left side with an endpoint abutting $L$, and similarly $\gamma_2$
is on the right.  Also,  $\gamma_1$ and $\rho\gamma_2$ are non-intersecting with adjacent endpoints as marked.}
  \label{fig:rsw-ref}
\end{figure}

We show next an inequality similar to \eqref{bf2} but for more general domains.
Let $\gamma_1$, $\gamma_2$ be paths as described in the caption
of Figure~\ref{fig:rsw-ref}, and consider the \rc\ measure,
denoted $\phi_{\gamma_1,\gamma_2}$, on the primal graph within the coloured
region $G(\gamma_1,\gamma_2)$ of the figure, with mixed wired/free boundary conditions
obtained by identifying all points on $\gamma_1$, and similarly on $\gamma_2$
(these two sets are not wired together as one).  For readers who prefer words to pictures:
$\g_1$ (\resp, $\g_2$) is a path on the left (\resp, right) of the line
$L$ of the figure, with
exactly one endpoint adjacent to $L$;
reflection in $L$ is denoted $\rho$; $\g_1$ and $\rho\g_2$ (and hence $\g_2$ and $\rho\g_1$ also)
do not intersect, and their other endpoints are adjacent in the sense of the figure.

Writing $\{\gamma_1 \lra \gamma_2\}$ for the event that there is an
open path in $G(\gamma_1, \gamma_2)$
from $\gamma_1$ to $\gamma_2$, we have by duality that
\begin{equation}\label{bf3}
\phi_{\gamma_1,\gamma_2}(\gamma_1 \lra \gamma_2) + \phi^*_{\gamma_1,\gamma_2}(\rho\gamma_1
\lra^* \rho\gamma_2) = 1,
\end{equation}
where $\phi^*_{\gamma_1,\gamma_2}$ is the \rc\ measure
on the dual of the graph within $G(\g_1,\g_2)$
and $\lra^*$ denotes the existence of an open dual connection.
Now, $\phi^*_{\gamma_1,\gamma_2}$
has a mixed boundary condition in which all vertices of $\rho\gamma_1 \cup \rho\gamma_2$
are identified. Since the number of clusters with this wired boundary condition differs
by at most 1 from that in which $\rho\gamma_1$ and $\rho\gamma_2$ are \emph{separately} wired,
the Radon--Nikod\'ym derivative of $\phi^*_{\gamma_1,\gamma_2}$
with respect to $\rho \phi_{\gamma_1,\gamma_2}$ takes values in the interval $[q^{-1},q]$.
Therefore,
$$
\phi^*_{\gamma_1,\gamma_2}(\rho\gamma_1
\lra^* \rho\gamma_2) \le {q^2}\phi_{\gamma_1,\gamma_2}(\gamma_1
\lra \gamma_2).
$$
By \eqref{bf3},
\begin{equation}\label{bf4}
\phi_{\gamma_1,\gamma_2}(\gamma_1 \lra \gamma_2) \ge \frac 1{1+q^2}.
\end{equation}

\smallskip\noindent
\emph{Step 2, crossing rectangles}.
We show next how \eqref{bf4} may be used to prove Proposition
\ref{thm:rsw-rc}.
Let $S=S_1 \cup S_2$, with $S_1=[0,n)\times[0,n)$ and $S_2 = [\frac12 n,\frac32 n)\times[0,n)$, as illustrated in
Figure~\ref{fig:rsw-rcm}. Let $A$ be the (increasing) event that: $S_1 \cup S_2$ contains some open cluster $C$ that
contains both a horizontal crossing of $S_1$ and a vertical crossing of $S_2$. We claim that
\begin{equation}\label{bf5}
\phi^\tp_{\psd,m}(A) \ge \frac{c_1^2}{2(1+q^2)},\qq m \ge \tfrac32 n,
\end{equation}
with $c_1$ as in \eqref{bf2}.
The proposition follows from \eqref{bf5}
since, by positive association and \eqref{bf2},
\begin{align*}
\phi_{\psd,m}^\tp\left[\Ch\bigl([0,\tfrac32 n)\times[0, n)\bigr)\right]
&\ge \phi^\tp_{\psd,m}\left[A \cap \Ch(S_2)\right]\\
&\ge \phi^\tp_{\psd,m}(A) \phi^\tp_{\psd,m}\left[\Ch(S_2)\right] \\
&\ge \frac {c_1^3}{2(1+q^2)}.
\end{align*}
We prove \eqref{bf5} next.

Let $\ell$ be the line-segment $[\frac12 n,n)\times\{0\}$,
and let $\Cv^\ell(S_2)$ be the event of a vertical crossing of
$S_2$ whose only endpoint on the $x$-axis lies in $\ell$. By a
symmetry of the \rc\ model, and \eqref{bf2},
\begin{equation}\label{g14}
\phi_{\psd,m}^\tp[\Cv^\ell(S_2)] \ge \tfrac12 \phi^\tp_{\psd,q}[\Cv(S_2)] \ge \tfrac12 c_1.
\end{equation}
On the event $\Ch(S_1)$ (\resp,
$\Cv^\ell(S_2)$) let $\Gamma_1$ (\resp, $\Gamma_2$) be the
highest (\resp, rightmost) crossing of the required type.
The paths $\Gamma_i$ may be used to construct the coloured region of Figure~\ref{fig:rsw-rcm}:
$L$ is a line in whose reflection the primal and dual lattices are interchanged; the reflections $\rho\Gamma_i$
of the $\Gamma_i$ frame a region bounded by subsets $\gamma_i$ of $\Gamma_i$ and
their reflections $\rho\gamma_i$. The situation is generally more complicated than the illustration
in the figure since the $\Gamma_i$ can wander around $S$  (see \cite{Beffara_Duminil}),
but the essential ingredients of the proof are clearest thus.

\begin{figure}[t]
\centering
    \includegraphics[width=0.8\textwidth]{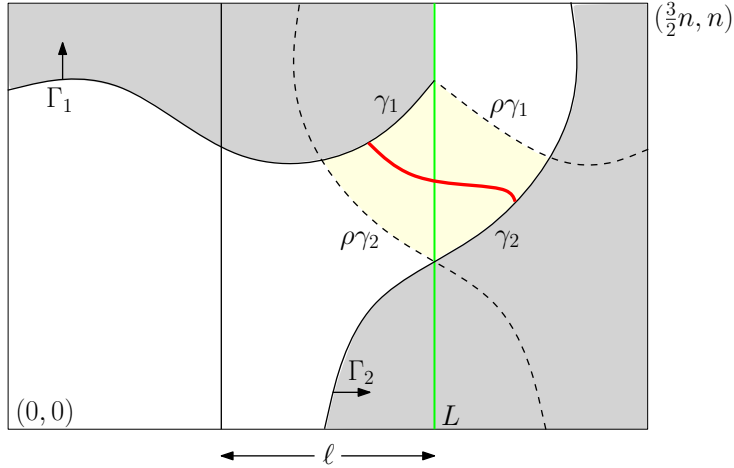}
  \caption{The path $\Gamma_1$ is the highest crossing of $S_1$, and $\Gamma_2$ is the rightmost
crossing of $S_2$ starting in $\ell$.
The coloured region is framed by the $\Gamma_i$ and their reflections in $L$.
This is the simplest situation and there are more complex, depending on the $\Gamma_i$.}
  \label{fig:rsw-rcm}
\end{figure}

Let $I = \{\Gamma_1 \cap \Gamma_2 \ne\es\}$, so that
\begin{align}
\phi^\tp_{\psd,m}(A) &\ge \phi^\tp_{\psd,m}\bigl[\Ch(S_1) \cap \Cv^\ell(S_2) \cap I\bigr]\nonumber\\
&\hskip2cm + \phi^\tp_{\psd,m}\bigl[A \cap \Ch(S_1) \cap \Cv^\ell(S_2) \cap \compl{I}\bigr].
\label{bf6}
\end{align}
On the event $\Ch(S_1) \cap \Cv^\ell(S_2) \cap \compl{I}$,
$$
\phi^\tp_{\psd,m}(A \mid \Gamma_1,\Gamma_2)
\ge \phi^\tp_{\psd,m}\bigl(\gamma_1 \lra \gamma_2 \text{ in } G(\gamma_1,\gamma_2)
\bigmid \Gamma_1 ,\Gamma_2\bigr).
$$
Since $\{\gamma \lra \gamma_2\text{ in } G(\gamma_1,\gamma_2)\}$ is an increasing event,
the right side is no larger if we augment
the conditioning with the event that all edges of $S_1$ strictly above $\Gamma_1$ (and $S_2$ to
the right of $\Gamma_2$) are closed.  It may then be seen that
\begin{equation}\label{G201}
\phi^\tp_{\psd,m}\bigl(\gamma_1 \lra \gamma_2 \text{ in } G(\gamma_1,\gamma_2)
\bigmid \Gamma_1 ,\Gamma_2\bigr)
\ge \phi_{\gamma_1,\gamma_2}(\gamma_1 \lra \gamma_2).
\end{equation}
This follows from \eqref{G11-} by conditioning on the configuration off
$G(\g_1,\g_2)$. By \eqref{g14}--\eqref{bf6}, \eqref{bf4}, and positive association,
\begin{align*}
\phi^\tp_{\psd,m}(A) &\ge \frac1{1+q^2} \phi^\tp_{\psd,m}\bigl(\Ch(S_1) \cap \Cv^\ell(S_2)\bigr)\\
&\ge
\frac{c_1^2}{2(1+q^2)},
\end{align*}
and \eqref{bf5} follows.

We repeat the need for care in deducing \eqref{G201} in general, since
the picture can be more complicated than indicated in Figure~\ref{fig:rsw-rcm}.

\subsection{Proof of the critical value}\label{sec:rcm-pf}

The inequality $\pc\ge \psd$  follows from the stronger statement $\theta^\tf(\psd)=0$,
and has been known since \cite{G93, Welsh93}. Here is a brief explanation.
We have that
$\phi^\tf_{p,q}$ is ergodic and has the so-called finite-energy property.
By the Burton--Keane uniqueness
theorem \cite{BK89}, the number of infinite open clusters
is either a.s.\ 0 or a.s.\ 1. If $\theta^\tf(\psd)>0$, then a contradiction
follows by duality, as in the case of percolation. Hence, $\theta^\tf(\psd)=0$,
and therefore $\pc\ge \psd$. The details may be found in \cite[Sect. 6.2]{Grimmett_RCM}.

It suffices then to show that $\theta^\tw(p)>0$ for $p > \psd$, since this implies $\pc\le \psd$.
The argument of \cite{Beffara_Duminil} follows the classic route
for percolation, but with two significant twists.
The first of these is the use of a sharp-threshold theorem for the \rc\ measure,
combined with the uniform estimate of
Proposition~\ref{thm:rsw-rc}, to show that, when $p>\psd$, the chances of box-crossings are near
to~1.\looseness=-1

\begin{prop}\label{thm:rsw-rc+}
Let  $p>\psd$. For  integral $\beta > \alpha\ge 2$,
there exist $a, b>0$ such that
$$
\phi^\tp_{p,\beta n}\left[\Ch\bigl([0,\alpha n] \times [0,n)\bigr)\right] \ge 1-a n^{-b},
\qquad n \ge 1.
$$
\end{prop}

\begin{proof}[Outline proof]
Recall first  the remark after Proposition~\ref{thm:rsw-rc}
that $\phi^\tp_{\psd,m}$ has the $\a$-\bxp.
Now some history. In proving that $\pc=\frac12$ for bond percolation on $\LL^2$, Kesten used
a geometrical argument to derive a sharp-threshold statement for box-crossings
along the following lines: since crossings of rectangles with given aspect-ratio
have probabilities bounded away from 0 when $p=\frac12$,
they have probability close to 1 when $p>\frac12$.  Kahn, Kalai, and Linial (KKL) \cite{KKL}
derived a general approach to sharp-thresholds of probabilities $\PP_{\frac12}(A)$
of increasing events $A$ under the product measure $\PP_{\frac12}$,
and this was extended later to more general product measures (see \cite[Chap. 4]{Grimmett_Graphs}
and \cite{KalS} for general accounts). Bollob\'as and Riordan
\cite{BR06} observed that the KKL method could be used instead of Kesten's geometrical method.
The KKL method works best for events and measures with a certain symmetry, and it is explained
in \cite{BR06} how this may be adapted for percolation box-crossings.

The KKL theorem was extended in \cite{GG_inf} to measures satisfying the FKG
lattice condition (such as, for example, \rc\ measures). The
symmetrization argument of \cite{BR06}
may be adapted to the \rc\ model with periodic boundary
condition (since the measure is translation-invariant), and the current proposition is a consequence.

See \cite{Beffara_Duminil} for the details,
and \cite[Sect. 4.5]{Grimmett_Graphs} for an account of the KKL method, with proofs
and references.
\end{proof}

\begin{figure}[t]
\centering
    \includegraphics[width=0.5\textwidth]{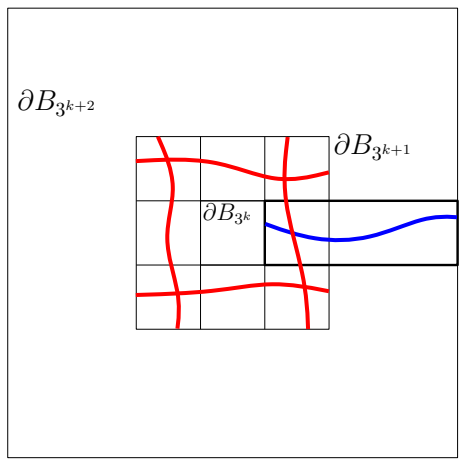}
  \caption{If the four red box-crossings exist, as well as the blue box-crossing,
then the event $A_k$ occurs.}
  \label{fig:boxes}
\end{figure}

Let $p > \psd$, and consider the annulus
$\sA_k = B_{3^{k+1}}  \setminus B_{3^k}$.
Let $A_k$ be the event that $\sA_k$ contains an open cycle $C$ with $0$ in its inside,
and in addition there is an open path from $C$ to the  boundary
$\pd B_{3^{k+2}}$.
We claim that there exist $c,d>0$, independent of $k$, such that
\begin{equation}\label{bf7}
\phi^1_{p,3^{k+2}}(A_k) \ge 1- ce^{-dk}, \qquad k \ge 1.
\end{equation}
This is proved as follows. The event $A_k$ occurs whenever the two rectangles
$$
[-3^{k+1},-3^k) \times [-3^{k+1},3^{k+1}], \quad (3^k,3^{k+1}] \times [-3^{k+1},3^{k+1}]
$$
are crossed vertically, and in addition the three rectangles
\begin{gather*}
[-3^{k+1},3^{k+1}] \times [-3^{k+1},-3^{k}), \quad [-3^{k+1},3^{k+1}] \times (3^{k},3^{k+1}],\\
[3^k,3^{k+2}] \times [-3^k,3^k]
\end{gather*}
are crossed horizontally. See Figure~\ref{fig:boxes}.
Each of these five rectangles has shorter dimension at least
$3^k$\vadjust{\eject}
and longer dimension not exceeding $3^{k+2}$.
By Proposition~\ref{thm:rsw-rc+} and the invariance of
$\phi^\tp_{p,3^{k+2}}$ under rotations and translations, each of these five events has
$\phi^\tp_{p,3^{k+2}}$-probability
at least $1-a3^{-bk}$ for suitable $a,b>0$. By stochastic
ordering \eqref{G11-} and positive association,
$$
\phi^1_{p,3^{k+2}}(A_k) \ge \phi^\tp_{p,3^{k+2}}(A_k)
\ge (1-a3^{-bk})^5,
$$
and \eqref{bf7} is proved.

Recall the weak limit $\fpq^1=\lim_{k\to\oo} \phi^\tw_{p,3^k}$. The
events $A_k$ have been defined in such a way that,
on the event $I_K = \bigcap_{k=K}^\oo A_k$, there exists an
infinite open cluster.
It suffices then to show that $\fpq^1(I_K)>0$ for large $K$.
Now,
\begin{equation}\label{bf8}
\fpq^1\left(\bigcap_{k=K}^m A_k\right) =
\fpq^1(A_m) \prod_{k=K}^{m-1}
\fpq^1\left(A_k \,\left|\, \bigcap_{l=k+1}^m A_l \right.\right).
\end{equation}

Let $\Ga_{k}$ be the outermost open cycle in $\sA_{k}$, whenever
this exists. The conditioning on the right side of
\eqref{bf8} amounts to the existence of $\Ga_{k+1}$ together with the
event $\{\Ga_{k+1}\lra \pd B_{3^{k+2}}\}$, in addition to some further information,
$I$ say, about the configuration  outside $\Ga_{k+1}$. For any appropriate cycle $\g$, the event
$\{\Gamma_{k+1}=\g\}\cap\{\g \lra \pd B_{3^{k+2}}\}$
is defined in terms of the states of edges of $\g$ and outside $\g$.
On this event,
$A_k$ occurs if and only if $A_k(\g):= \{\Ga_k \text{ exists}\}\cap\{\Ga_k \lra \g\}$ occurs.
The latter event is measurable on the states of edges inside $\g$, and the appropriate
conditional \rc\ measure is that with wired boundary condition inherited from $\g$,
denoted $\phi^1_\g$. (We
have used the fact that the cluster-count inside $\g$ is not changed by $I$.)
Therefore, the term in the product
on the right side of \eqref{bf8} equals the average over $\g$ of
$$
\fpq^1\left(A_k(\g) \,\Big|\, \{\Ga_{k+1}=\g\}\cap \{\g\lra\pd B_{3^{k+2}}\}\cap I\right) =
\phi^1_{\g}(A_k(\g)).
$$
Let $\De = \pd B_{3^{k+2}}$.
Since $\phi_{\De}^1 \lest \phi^1_{\g}$ and $A_k(\De) \subseteq A_k(\g)$,
$$
\phi^1_\g(A_k(\g)) \ge \phi^1_\De(A_k(\De)) = \phi^1_{p,3^{k+2}}(A_k).
$$
In conclusion,
\begin{equation}\label{g13}
\fpq^1\left(A_k \,\left|\, \bigcap_{l=k+1}^m A_l \right.\right)
\ge \phi^1_{p,3^{k+2}}(A_k).
\end{equation}

By \eqref{bf7}--\eqref{g13},
$$
\fpq^1\left(\bigcap_{k=K}^m A_k\right) \ge
\fpq^1(A_m) \prod_{k=K}^{m-1} (1-ce^{-dk}).
$$
By \eqref{bf9}, the \bxp, and positive association
(as in the red paths of Figure~\ref{fig:boxes}), there exists $c_2>0$ such
that $\fpq^1(A_m)\ge c_2$ for all $m \ge 1$. Hence,
$$
\fpq^1(I_K) = \lim_{m\to\oo} \fpq^1\left(\bigcap_{k=K}^m A_k\right)
\ge c_2\prod_{k=K}^\oo (1-ce^{-dk}),
$$
which is strictly positive for large $K$. Therefore $\th^1(p,q)>0$,
and the theorem is proved.

\section*{Acknowledgements}
The author is grateful to Ioan Manolescu for many discussions concerning the material in
Section~\ref{sec:bp}, and for his detailed comments
on a draft of these notes. He thanks the organizers and \lq students' at the 2011
Cornell Probability Summer
School and the 2011 Reykjavik Summer School in Random Geometry.
This work was supported in part by the EPSRC under grant EP/103372X/1.

\bibliographystyle{amsplain}

\providecommand{\bysame}{\leavevmode\hbox to3em{\hrulefill}\thinspace}
\providecommand{\MR}{\relax\ifhmode\unskip\space\fi MR }
\providecommand{\MRhref}[2]{%
  \href{http://www.ams.org/mathscinet-getitem?mr=#1}{#2}
}
\providecommand{\href}[2]{#2}
\begin{thebibliography}{}

\end{thebibliography}


\begin{thebibliography}{10}

\bibitem{ACCN}
\textsc{M.~Aizenman, J.~T. Chayes, L.~Chayes, and C.~M. Newman}, \emph{{Discontinuity of
  the magnetization in one-dimensional $1/|x-y|^2$ Ising and Potts models}}, J.
  Statist. Phys. \textbf{50} (1988), 1--40.
\MR{0939480}

\bibitem{BDGS}
\textsc{R.~Bauerschmidt, H.~Duminil-Copin, J.~Goodman, and G.~Slade}, \emph{Lectures on
  self-avoiding-walks}, {Probability and Statistical Mechanics in Two and More
  Dimensions} (D.~Ellwood, C.~M. Newman, V.~Sidoravicius, and W.~Werner, eds.),
  CMI/AMS publication, 2011.

\bibitem{Baxter_book}
\textsc{R.~J. Baxter}, \emph{{Exactly Solved Models in Statistical Mechanics}}, Academic
  Press, London, 1982.
\MR{0690578}

\bibitem{Baxter_399}
\textsc{R.~J. Baxter and I.~G. Enting}, \emph{399th solution of the {I}sing model}, J.
  Phys. A: Math. Gen. \textbf{11} (1978), 2463--2473.

\bibitem{BGG}
\textsc{N.~R. Beaton, J.~de Gier, and A.~G. Guttmann}, \emph{The critical fugacity for
  surface adsorption of {SAW} on the honeycomb lattice is $1+\sqrt 2$}, Commun.
  Math. Phys., \url{http://arxiv.org/abs/1109.0358}.

\bibitem{Beffara2010}
\textsc{V.~Beffara}, \emph{{SLE} and other conformally invariant objects}, {Probability
  and Statistical Mechanics in Two and More Dimensions} (D.~Ellwood, C.~M.
  Newman, V.~Sidoravicius, and W.~Werner, eds.), CMI/AMS publication, 2011.

\bibitem{Beffara_Duminil}
\textsc{V.~Beffara and H.~{Duminil-Copin}}, \emph{The self-dual point of the
  two-dimensional random-cluster model is critical for $q \geq 1$}, Probab. Th.
  Rel. Fields, \url{http://arxiv.org/abs/1006.5073}.

\bibitem{BDS2}
\textsc{V.~Beffara, H.~{Duminil-Copin}, and S.~Smirnov}, \emph{Parafermions in the
  random-cluster model},  (2011).

\bibitem{BR06}
\textsc{B.~Bollob\'as and O.~Riordan}, \emph{{A short proof of the Harris--Kesten
  theorem}}, J. Lond. Math. Soc. \textbf{38} (2006), 470--484.
\MR{2239042}

\bibitem{BolRio}
\textsc{B.~Bollob\'as and O.~Riordan}, \emph{Percolation}, Cambridge University Press, Cambridge, 2006.
\MR{2283880}

\bibitem{BK89}
\textsc{R.~M. Burton and M.~Keane}, \emph{Density and uniqueness in percolation},
  Commun. Math. Phys. \textbf{121} (1989), 501--505.
\MR{0990777}

\bibitem{CN06}
\textsc{F.~Camia and C.~M. Newman}, \emph{Two-dimensional critical percolation}, Commun.
  Math. Phys. \textbf{268} (2006), 1--38.
\MR{2249794}

\bibitem{Cardy}
\textsc{J.~Cardy}, \emph{Critical percolation in finite geometries}, J. Phys. A: Math.
  Gen. \textbf{25} (1992), L201--L206.
\MR{1151081}

\bibitem{Chelkak-Smirnov2}
\textsc{D.~Chelkak and S.~Smirnov}, \emph{{Universality in the 2D Ising model and
  conformal invariance of fermionic observables}}, Invent. Math.,
  \url{http://arxiv.org/abs/0910.2045}.
\MR{2275653}

\bibitem{Dum-S}
\textsc{H.~Duminil-Copin and S.~Smirnov}, \emph{The connective constant of the honeycomb
  lattice equals $\sqrt{2+\sqrt2}$}, Ann. Math.,
  \url{http://arxiv.org/abs/1007.0575}.

\bibitem{DGKLP}
\textsc{B.~Dyhr, M.~Gilbert, T.~Kennedy, G.~F. Lawler, and S.~Passon}, \emph{The
  self-avoiding walk spanning a strip}, J. Statist. Phys. \textbf{144} (2011),
  1--22.

\bibitem{MF66}
\textsc{M.~E. Fisher}, \emph{{On the dimer solution of planar Ising models}}, J. Math.
  Phys. \textbf{7} (1966), 1776--1781.

\bibitem{GG_inf}
\textsc{B.~T. Graham and G.~R. Grimmett}, \emph{Influence and sharp-threshold theorems
  for monotonic measures}, Ann. Probab. \textbf{34} (2006), 1726--1745.
\MR{2271479}

\bibitem{G93}
\textsc{G.~R. Grimmett}, \emph{The stochastic random-cluster process and the uniqueness
  of random-cluster measures}, Ann. Probab. \textbf{23} (1995), 1461--1510.
\MR{1379156}

\bibitem{Grimmett_Percolation}
\textsc{G.~R. Grimmett}, \emph{Percolation}, 2nd ed., Springer, Berlin, 1999.
\MR{1707339}

\bibitem{Grimmett_RCM}
\textsc{G.~R. Grimmett}, \emph{{The Random-Cluster Model}}, Springer, Berlin, 2006.
\MR{2243761}

\bibitem{Grimmett_Graphs}
\textsc{G.~R. Grimmett}, \emph{{Probability on Graphs}}, Cambridge University Press, Cambridge,
  2010, \url{http://www.statslab.cam.ac.uk/~grg/books/pgs.html}.
\MR{2723356}

\bibitem{GM1}
\textsc{G.~R. Grimmett and I.~Manolescu}, \emph{Inhomogeneous bond percolation on the
  square, triangular, and hexagonal lattices}, Ann. Probab.,
  \url{http://arxiv.org/abs/1105.5535}.
\eject

\bibitem{GM2}
\textsc{G.~R. Grimmett and I.~Manolescu}, \emph{Universality for bond percolation in two dimensions}, Ann.
  Probab., \url{http://arxiv.org/abs/1108.2784}.

\bibitem{GM3}
\textsc{G.~R. Grimmett and I.~Manolescu}, \emph{Bond percolation on isoradial graphs},  (2011), in preparation.

\bibitem{HM}
\textsc{J.~M. Hammersley and W.~Morton}, \emph{{Poor man's Monte Carlo}}, J. Roy.
  Statist. Soc. B \textbf{16} (1954), 23--38.
\MR{0064475}

\bibitem{HW62}
\textsc{J.~M. Hammersley and D.~J.~A. Welsh}, \emph{Further results on the rate of
  convergence to the connective constant of the hypercubical lattice}, Quart.
  J. Math. Oxford \textbf{13} (1962), 108--110.
\MR{0139535}

\bibitem{HS}
\textsc{T.~Hara and G.~Slade}, \emph{Mean-field critical behaviour for percolation in
  high dimensions}, Commun. Math. Phys. \textbf{128} (1990), 333--391.
\MR{1043524}

\bibitem{HS94}
\textsc{T.~Hara and G.~Slade}, \emph{Mean-field behaviour and the lace expansion}, {Probability and
  Phase Transition} (G.~R. Grimmett, ed.), Kluwer, 1994, pp.~87--122.
\MR{1283177}

\bibitem{HKW}
\textsc{D.~Hintermann, H.~Kunz, and F.~Y. Wu}, \emph{Exact results for the {P}otts model
  in two dimensions}, J. Statist. Phys. \textbf{19} (1978), 623--632.
\MR{0521142}

\bibitem{Hol}
\textsc{R.~Holley}, \emph{Remarks on the {FKG} inequalities}, Commun. Math. Phys.
  \textbf{36} (1974), 227--231.
\MR{0341552}

\bibitem{BDH}
\textsc{B.~D. Hughes}, \emph{{Random Walks and Random Environments; Volume I, Random
  Walks}}, Oxford University Press, Oxford, 1996.

\bibitem{JenGutt}
\textsc{I.~Jensen and A.~J. Guttman}, \emph{Self-avoiding walks, neighbour-avoiding
  walks and trails on semiregular lattices}, J. Phys. A: Math. Gen. \textbf{31}
  (1998), 8137--8145.
\MR{1651493}

\bibitem{KKL}
\textsc{J.~Kahn, G.~Kalai, and N.~Linial}, \emph{The influence of variables on {B}oolean
  functions}, Proceedings of 29th Symposium on the Foundations of Computer
  Science, 1988, pp.~68--80.

\bibitem{KalS}
\textsc{G.~Kalai and S.~Safra}, \emph{Threshold phenomena and influence}, {Computational
  Complexity and Statistical Physics} (A.~G. Percus, G.~Istrate, and C.~Moore,
  eds.), Oxford University Press, New York, 2006, pp.~25--60.
\MR{2208732}

\bibitem{Ken}
\textsc{A.~E. Kennelly}, \emph{The equivalence of triangles and three-pointed stars in
  conducting networks}, Electrical World and Engineer \textbf{34} (1899),
  413--414.

\bibitem{Ken02}
\textsc{R.~Kenyon}, \emph{An introduction to the dimer model}, School and Conference on
  Probability Theory, Lecture Notes Series, vol.~17, ICTP, Trieste, 2004, {\tt
  \url{http://publications.ictp.it/lns/vol17/vol17toc.html}}, pp.~268--304.
\MR{2198850}

\bibitem{Kesten_book}
\textsc{H.~Kesten}, \emph{{Percolation Theory for Mathematicians}}, Birkh\"auser,
  Boston, 1982.
\MR{0692943}

\bibitem{Kes87a}
\textsc{H.~Kesten}, \emph{A scaling relation at criticality for {$2D$}-percolation},
  {Percolation Theory and Ergodic Theory of Infinite Particle Systems}
  (H.~Kesten, ed.), {The IMA Volumes in Mathematics and its Applications},
  vol.~8, Springer, New York, 1987, pp.~203--212.

\bibitem{Kesten87}
\textsc{H.~Kesten}, \emph{Scaling relations for {2D}-percolation}, Commun. Math. Phys.
  \textbf{109} (1987), 109--156.
\eject

\bibitem{K11}
\textsc{M.~Klazar}, \emph{{On the theorem of Duminil-Copin and Smirnov about the number
  of self-avoiding walks in the hexagonal lattice}},  (2011),
  \url{http://arxiv.org/abs/1102.5733}.

\bibitem{Kot-S}
\textsc{R.~Koteck\'y and S.~Shlosman}, \emph{First order phase transitions in large
  entropy lattice systems}, Commun. Math. Phys. \textbf{83} (1982), 493--515.
\MR{0649814}

\bibitem{KN}
\textsc{G.~Kozma and A.~Nachmias}, \emph{Arm exponents in high dimensional percolation},
  J. Amer. Math. Soc. \textbf{24} (2011), 375--409.
\MR{2748397}

\bibitem{LMMRS}
\textsc{L.~Laanait, A.~Messager, S.~Miracle-Sol\'e, J.~Ruiz, and S.~Shlosman},
  \emph{{Interfaces in the Potts model I: Pirogov--Sinai theory of the
  Fortuin--Kasteleyn representation}}, Commun. Math. Phys. \textbf{140} (1991),
  81--91.
\MR{1124260}

\bibitem{LMR}
\textsc{L.~Laanait, A.~Messager, and J.~Ruiz}, \emph{{Phase coexistence and surface
  tensions for the Potts model}}, Commun. Math. Phys. \textbf{105} (1986),
  527--545.
\MR{0852089}

\bibitem{Law11}
\textsc{G.~F. Lawler}, \emph{{Scaling limits and the Schramm--Loewner evolution}},
  Probability Surveys \textbf{8} (2011).

\bibitem{SLW04}
\textsc{G.~F. Lawler, O.~Schramm, and W.~Werner}, \emph{Conformal invariance of planar
  loop-erased random walks and uniform spanning trees}, Ann. Probab.
  \textbf{32} (2004), 939--995.
\MR{2044671}

\bibitem{SLW04b}
\textsc{G.~F. Lawler, O.~Schramm, and W.~Werner}, \emph{On the scaling limit of planar self-avoiding walk}, Proc.
  Symposia Pure Math. \textbf{72} (2004), 339--365.
\MR{2112127}

\bibitem{MS}
\textsc{N.~Madras and G.~Slade}, \emph{{The Self-Avoiding Walk}}, Birkh\"auser, Boston,
  1993.
\MR{1197356}

\bibitem{Nienhuis}
\textsc{B.~Nienhuis}, \emph{Exact critical point and exponents of {O}$(n)$ models in two
  dimensions}, Phys. Rev. Lett. \textbf{49} (1982), 1062--1065.
\MR{0675241}

\bibitem{Nolin}
\textsc{P.~Nolin}, \emph{Near-critical percolation in two dimensions}, Electron. J.
  Probab. \textbf{13} (2008), 1562--1623.
\MR{2438816}

\bibitem{OnsI}
\textsc{L.~Onsager}, \emph{{Crystal statistics. I. A two-dimensional model with an
  order--disorder transition}}, Phys. Rev. \textbf{65} (1944), 117--149.
\MR{0010315}

\bibitem{Perk-AY}
\textsc{J.~H.~H. Perk and H.~Au-Yang}, \emph{{Yang--Baxter} equation}, Encyclopedia of
  Mathematical Physics (J.-P. Fran\c{c}oise, G.~L. Naber, and S.~T. Tsou,
  eds.), vol.~5, Elsevier, 2006, pp.~465--473.

\bibitem{Russo}
\textsc{L.~Russo}, \emph{A note on percolation}, Z. Wahrsch'theorie verw. Geb.
  \textbf{43} (1978), 39--48.
\MR{0488383}

\bibitem{Sch00}
\textsc{O.~Schramm}, \emph{Scaling limits of loop-erased walks and uniform spanning
  trees}, Israel J. Math. \textbf{118} (2000), 221--288.
\MR{1776084}

\bibitem{Sch06}
\textsc{O.~Schramm}, \emph{Conformally invariant scaling limits: an overview and collection
  of open problems}, {Proceedings of the International Congress of
  Mathematicians, Madrid} (M.\ Sanz-Sol\'e \textit{et al.}, ed.), vol.~{I},
  European Mathematical Society, Zurich, 2007, pp.~513--544.

\bibitem{SS09}
\textsc{O.~Schramm and S.~Sheffield}, \emph{{Contour lines of the two-dimensional
  discrete Gaussian free field}}, Acta Math. \textbf{202} (2009), 21--137.
\MR{2486487}

\bibitem{Seymour-Welsh}
\textsc{P.~D. Seymour and D.~J.~A. Welsh}, \emph{Percolation probabilities on the square
  lattice}, Ann. Discrete Math. \textbf{3} (1978), 227--245.
\MR{0494572}

\bibitem{slade10}
\textsc{G.~Slade}, \emph{The self-avoiding walk: a brief survey}, {Surveys in Stochastic
  Processes} (J.~Blath, P.~Imkeller, and S.~Roelly, eds.), European
  Mathematical Society, 2010, {Proceedings of the 33rd SPA Conference, Berlin,
  2009}.

\bibitem{Smirnov}
\textsc{S.~Smirnov}, \emph{Critical percolation in the plane: conformal invariance,
  {Cardy's} formula, scaling limits}, C. R. Acad. Sci. Paris Ser. I Math.
  \textbf{333} (2001), 239--244.
\MR{1851632}

\bibitem{Smi07}
\textsc{S.~Smirnov}, \emph{{Towards conformal invariance of {2D} lattice models}},
  {Proceedings of the {I}nternational {C}ongress of {M}athematicians, {M}adrid,
  2006} (M.~Sanz-Sol\'e \textit{et al.}, ed.), vol.~{II}, European {M}athematical
  {S}ociety, Zurich, 2007, pp.~1421--1452.
\MR{2275653}

\bibitem{Smirnov-Werner}
\textsc{S.~Smirnov and W.~Werner}, \emph{Critical exponents for two-dimensional
  percolation}, Math. Res. Lett. \textbf{8} (2001), 729--744.
\MR{1879816}

\bibitem{Sun11}
\textsc{N.~Sun}, \emph{Conformally invariant scaling limits in planar critical
  percolation}, Probability Surveys \textbf{8} (2011), 155--209.

\bibitem{Sykes_Essam}
\textsc{M.~F. Sykes and J.~W. Essam}, \emph{Some exact critical percolation
  probabilities for site and bond problems in two dimensions}, J. Math. Phys.
  \textbf{5} (1964), 1117--1127.
\MR{0164680}

\bibitem{Welsh93}
\textsc{D.~J.~A. Welsh}, \emph{Percolation in the random-cluster process}, J. Phys. A:
  Math. Gen. \textbf{26} (1993), 2471--2483.
\MR{1234408}

\bibitem{WW_park_city}
\textsc{W.~Werner}, \emph{Lectures on two-dimensional critical percolation}, Statistical
  Mechanics (S.~Sheffield and T.~Spencer, eds.), vol.~16, IAS--Park City, 2007,
  pp.~297--360.
\MR{2523462}

\bibitem{Werner_SMF}
\textsc{W.~Werner}, \emph{{Percolation et Mod\`ele d'Ising}}, Cours Specialis\'es,
  vol.~16, Soci\'et\'e Math\'ematique de France, Paris, 2009.

\end{thebibliography}

\providecommand{\bysame}{\leavevmode\hbox to3em{\hrulefill}\thinspace}
\providecommand{\MR}{\relax\ifhmode\unskip\space\fi MR }
\providecommand{\MRhref}[2]{%
  \href{http://www.ams.org/mathscinet-getitem?mr=#1}{#2}
}
\providecommand{\href}[2]{#2}

\end{document}